\documentclass[11pt]{article}

\usepackage[margin=1in]{geometry} 
\usepackage{amsmath,amssymb,amsfonts,amsthm}
\usepackage{graphicx}
\usepackage{float}
\usepackage{xcolor}
\usepackage[authoryear, round]{natbib}
\usepackage{hyperref}
\usepackage{makecell}
\hypersetup{colorlinks,citecolor=blue,linkcolor=blue,breaklinks=true}
\usepackage{epstopdf}
\usepackage{yhmath} 
\usepackage{booktabs}
\usepackage{algorithm}
\usepackage{algpseudocode}
\usepackage{enumitem}
\usepackage{subcaption}
\usepackage{natbib}
\usepackage{ulem}
\allowdisplaybreaks
\usepackage{tikz}
\usetikzlibrary{shapes,arrows,positioning,backgrounds,shadows,calc}
\definecolor{myblue}{RGB}{43,87,154}
\definecolor{myorange}{RGB}{242,80,34}
\definecolor{mygreen}{RGB}{0,163,0}
\tikzstyle{arrow} = [thick,->,>=stealth]

\allowdisplaybreaks

\newtheorem{theorem}{Theorem}[section]
\newtheorem{lemma}{Lemma}[section]

\newtheorem{proposition}[theorem]{Proposition}
\newtheorem{definition}{Definition}[section]

\numberwithin{equation}{section}
\DeclareMathOperator{\diag}{diag}

\begin{document}
	
	\title{Dynamic Data Pricing: A Mean Field Stackelberg Game Approach}
	
	\author{
		Lijun Bo
		\thanks{lijunbo@xidian.edu.cn, School of Mathematics and Statistics, 
    Xidian University, Xi'an, 710126, China}
		\and
		Dongfang Yang
		\thanks{yangdf@stu.xidian.edu.cn, School of Mathematics and Statistics, 
    Xidian University, Xi'an, 710126, China}
		\and
		Shihua Wang\thanks{Corresponding author: wangshihua@xidian.edu.cn, School of Mathematics and Statistics, Xidian University, Xi'an, 710126, China}	
	}
	
	\date{}
	
	\maketitle
	
	\begin{abstract}
This paper studies the dynamic pricing mechanism for data products in demand-driven markets through a game-theoretic framework. We develop a three-tier Stackelberg game model to capture the hierarchical strategic interactions among key market entities: a single data buyer, an intermediary broker, and a competitive seller group. To characterize the temporal dynamics of data quality evolution, we establish a coupled system of stochastic differential equations (SDEs) where sellers' quality investments interact through mean-field effects. Given exogenous pricing policies, we derive approximate Nash equilibrium adjustment strategies for competitive sellers using the mean field game (MFG) approach. The broker’s optimal pricing strategy is subsequently established by solving a Stackelberg leadership problem, while the buyer’s procurement policy is determined through an optimal control formulation involving conditional mean-field forward-backward SDEs (FBSDEs). Under some regularity conditions, the proposed strategies are shown to collectively form an $(\epsilon_1, \epsilon_2, \epsilon_3)$-Stackelberg equilibrium.
    
	\vspace{0.1in}
	\end{abstract}
	
	\section{Introduction}
		In the era of digital economy, data is recognized an invaluable resource. Beyond its role as a raw material for product development, data is also directly monetizable as a product itself.  The expanding demand for data trading has led to the proliferation of dedicated platforms such as AWS Data Exchange, Dawex and BDEX. Transactions occur in multiple forms, primarily categorized as raw data and data products, the latter encompassing processed outputs such as query answers (e.g. \cite{KU15}) and machine learning models (e.g. \cite{LL21}). What fundamentally distinguishes data from traditional commodities are its inherent characteristics. Foremost among these is replicability: unlike physical goods, data can be reproduced at near-zero marginal cost. Additional features such as dynamicity, where the data value may evolve or decay over time, and composability, which enables the integration of datasets, further complicate its valuation and pricing. Despite extensive research on data pricing in dimensions such as truthfulness, revenue maximization, and privacy preservation (see \cite{P22} for a survey), a systematic and unified mechanism has not yet been established.

Data quality constitutes a key determinant in the pricing mechanisms of datasets, given its nature as a multidimensional construct. Through a two-stage survey and a two-phase sorting study, \cite{W96} develop a hierarchical framework for data quality. This process identified 15 key dimensions from an initial pool of 179 criteria. Based on previous research, \cite{B09} propose a fundamental set of data-quality dimensions, including accuracy, completeness, consistency, and timeliness. \cite{H15} propose a linear model that takes into account several  quality dimensions to determine the price of a dataset, including fixed cost, age, periodicity, volume, and accuracy. \cite{Y17} present a bi-level programming model for the data-pricing problem based on data quality with the objective of maximizing both platform profit and consumer utility. Two aspects of data quality are incorporated: multidimensionality and interactions between dimensions. \cite{CK19} provide different versions of the desired data products-machine learning models for quality differentiation. 
\cite{YZ19} develop a linear evaluation model that assesses data quality across multiple dimensions including accuracy, completeness, and redundancy, and subsequently determines price based on the resulting quality score.
	
As data freshness, a significant dimension of data quality, is inherently dynamic in nature, a well-designed pricing model must mirror this evolution. Several studies handle mechanisms for dynamic data pricing. \cite{NZ20} design a contextual dynamic pricing mechanism for online data market where a query may be sold to  different buyers at different time and the broker can adjust prices over time. \cite{XJ16} address the problem of revenue maximization for a data collector facing sequentially arriving sellers with unknown privacy valuations. 
Focusing on strategic data buyers who may continually submit low bids to depress prices, \cite{C22} explores the corresponding data pricing techniques to counteract such behaviors. \cite{ZA21} study a pricing problem in the context of fresh data trading, in which a destination user requests and pays for fresh data updates from a source provider, and data freshness is captured by the age of information (AoI) metric. 
		
Game theory serves as a primary analytical tool for formulating data pricing problems, effectively modeling the rational and strategic behaviors of participants in data markets (c.f. \cite{BW24}). The Stackelberg game, pioneered by \cite{V52},  has been extended to capture the interactions among multiple participants in data trading. In the Stackelberg game, the leader is endowed with the first-move priority, selects an optimal strategy by anticipating the follower's response. The follower chooses his own optimal response accordingly. In the bi-level programming model proposed by \cite{Y17}, the monopolistic platform owner acts as the leader, with data consumers as followers. \cite{XZ20} study pricing in car-sharing data markets and propose a three-layer Stackelberg game comprising a data seller, a service provider, and a data buyer. 
Related studies with comparable settings include \cite{ZA21} and \cite{LH23}. The demand for data has exploded explosively in recent decades. On this basis, the appropriate framework catering to a demand-driven market setting, in which buyers act as active participants in price formation can be explored as in \cite{AX21}, \cite{XH21} and \cite{BL24}. In particular,  \cite{BL24} incorporate the sellers' inner competition into the data pricing problem within a Stackelberg game framework.
		
The vast majority of the aforementioned game-theoretic studies were conducted in the context of a static economic setting. \cite{Y02} establishes a generalized framework for linear-quadratic (LQ) stochastic leader-follower differential games, contributing to the development of dynamic game theory in a continuous-time stochastic setting. Further advancing this field, \cite{BCS15} derive a maximum principle for the solution of Stackelberg differential games under both adapted open-loop and memoryless closed-loop information structures. In a dynamic context, \cite{SW16} develop a stochastic maximum principle for leader-follower differential games with asymmetric information. However, realistic data markets are often characterized by hierarchical interactions among a large number of agents. The MFG theory, introduced in the independent work of \cite{LL07} and \cite{HM06}, is used to analyze systems with a large population of rational individual agents interacting with each other. The framework has been successfully applied across diverse domains, including economics and finance (e.g. \cite{G16}), engineering (e.g. \cite{MC11}), social studies (e.g. \cite{BT16}) and among others. In contrast to the decentralized nature of MFGs, mean field type control (MFC) seeks a centralized strategy with the objective of steering the collective behavior toward a system-level optimum. The existing literature has explored MFC problems from various perspectives using different methodologies, such as the stochastic maximum principle for McKean–Vlasov dynamics (\cite{CD15}), models with common noise (\cite{G16}), and dynamics of mean-field BSDEs (\cite{LS19}). The integration of Stackelberg games with MFG or MFC problem has been preliminarily explored in the literature. \cite{NC12} handles a large-population linear-quadratic (LQ) leader-follower stochastic multi-agent system and establish an $(\epsilon_1, \epsilon_2)$-Stackelberg-Nash equilibrium for such systems. \cite{BC15, BC17} concern mean field Stackelberg games involving time-delay effects. \cite{MB18} examine linear-quadratic-Gaussian (LQG) mean field Stackelberg games with heterogeneous agents with a fixed-point method. \cite{HS21} analyze a LQG large-population system comprising three agent types: a major leader, minor leaders and minor followers. For further research on this field, one may also refer to \cite{LJ18},  \cite{FH20}, \cite{W25} and among others.
	
In the current paper, we examine the dynamic pricing problem for data and data products in demand-driven markets through the lens of game theory. A three-layer Stackelberg game framework is established, characterizing the hierarchical interactions among a data buyer, a broker, and $n$ competitive data sellers. In the demand-driven market scenario (c.f. \cite{H24}), consumers have evolved from passive recipients to active participants who can signal demand through bidding.  The data broker generates profit by transforming purchased raw data into processed products (e.g., query answers and machine learning models) for the buyer. A large number of sellers supply raw datasets to the broker  while competing through quality differentiation. Moreover, we assume that sellers possess the capability to adjust primary data quality through techniques such as outlier detection, data integration, deduplication, data cleaning, and noise injection. Quality adjustment incurs a cost and loss of privacy. In a similar spirit to the posted price mechanism in \cite{BL24}, we take into account the dynamic nature of the data. The data quality dynamics of sellers are modeled via a system of stochastic differential equations with both idiosyncratic and common noises. 
Under the given price, sellers supply data at varying quality levels to trade-off profits against associated costs in a finite horizon. As a dominant leader, the buyer optimizes the utility of the data products minus payment to the broker. The broker's objective is to maximize the average revenue generated from acquiring and processing data obtained from $n$ sellers. The mean-field output adjustment model in \cite{W25} with a hierarchical structure between a regulator and multiple firms shares structural similarities with ours. The dominant regulator effects the stochastic sticky price to optimize social cost under the firms' Nash equilibrium responses as followers. The solution follows a direct method: first, solving a centralized $N$-agent game/control problem, then deriving decentralized strategies via the mean-field approximation as $N$ tends to infinity. Different from the methods, we start by deriving an MFG formulation using conditional expectations w.r.t. the common noise. By tackling a stochastic control problem of a representative seller under the given price with stochastic maximum principle, we acquire an open loop solution. A fixed-point equation system derived from the consistency condition is solved, and then a set of decentralized strategies, which is verified to be the approximate Nash equilibrium for sellers, is constructed. Auxiliary optimal control problems governed by conditional mean-field FBSDEs are resolved via variational analysis, and asymptotically optimal strategies are designed for broker and buyer. By a rigorous analysis, the established strategies are shown to form an $(\epsilon_1, \epsilon_2, \epsilon_3)$-Stackelberg equilibrium for the sequential game. This multi‑precision equilibrium concept quantifies the approximate optimality at each layer and ensures that no player can improve its payoff more than a small error by unilateral deviation, thereby providing a robust solution concept for large‑scale hierarchical stochastic games.
		
The main contributions of this paper are twofold. First, we propose a tractable dynamic hierarchical data market model that integrates a realistic partial information structure, involving a buyer, a broker, and a large number of competitive data sellers, which is novel and rather complex in the existing literature. In the model, the leader influences followers solely by announcing a pricing rule that enters their objective functions, which differs from the standard Stackelberg setting.  Second, this setup leads to complex internal coupling structures within both the leader's control problem and the followers' game problem, rendering their analysis challenging. To overcome it, we develop the mean field approximation technique to construct an approximate Nash equilibrium among followers. For the control problem of leader (sub-leader), we adopt an indirect analytical routine: formulating an auxiliary control problem and deriving an optimal solution characterized by conditional mean field FBSDEs. Rather than directly decoupling this conditional mean-field system, the final solution is constructed by solving equivalent low-dimensional FBSDE systems via the Riccati equations, which are proven to be asymptotically optimal for the leader.

This paper is structured as follows. Section~\ref{sec:model} presents our dynamic pricing model and the underlying market structure. The analysis of the Stackelberg game using backward induction is documented in Sections \ref{sec:3} and \ref{sec:4}. Specifically, Section \ref{sec:3} examines a non-cooperative game among sellers, where we devise decentralized adjustment strategies via mean-field approximation and prove that the resulting strategy profile forms an approximate Nash equilibrium (or approximate Nash equilibrium). Section \ref{sec:4} derives the equilibrium pricing strategies for both the broker and the buyer. Finally, Section \ref{sec:numerical} presents numerical simulations to illustrate the behavior of the model.
	
\section{The Model and Problem Formulation}\label{sec:model}
    
Let $T\in(0,\infty)$ be the terminal horizon. Consider a demand-driven data market consisting of one buyer, one broker, and $n$ competitive sellers. The buyer proposes the demand for data products and the price he/she is willing to pay, denoted by $p^{\rm bu}=(p_t^{\rm bu})_{t\in[0,T]}$. The broker buys raw data sets at his intended price $p^{\rm br}=(p_t^{\rm br})_{t\in[0,T]}$ from $n$ sellers, processes the data sets into products, and then sells them to the buyer. The sellers who possess data of the same type can adjust the data quality to sell data in distinctive versions.
	
Let $ (\Omega, \mathcal{F}, \mathbb{P}) $ be a probability space with filtration $ \mathbb{F} = (\mathcal{F}_t)_{t \in[0,T]} $ satisfying the usual conditions, which supports $n+1$ independent Brownian motions (BMs) $ W^i = (W_t^i)_{t \in[0,T]}$ for $ i = 0,1, \ldots, n$.  Denote by $\mathbb{F}^0=(\mathcal{F}_t^0)_{t\in [0,T]}=\sigma(W_s^0; s\leq t)_{t\in [0, T]}$, i.e., the filtration generated by the common BM $W^0$. 
Let $\mathbb{L}_{\mathbb{F}}^2([0,T])$ be the space of $\mathbb{F}$-progressively measurable real-valued processes $x=(x_t)_{t\in[0,T]}$ satisfying $\|x\|_{2,T}^2:=\mathbb{E} [\int_0^T|x_t|^2 dt] < \infty$. Consider $\tau=(\tau_t^i)_{t\in[0,T]}\in\mathbb{U}^{\rm ad}:=\mathbb{L}_{\mathbb{F}}^2([0,T])$ is the strategy made to adjust the data quality by seller $i$. For $i=1,\ldots,n$, let $Q_t^i$ be the data quality of the $i$-th seller at time $t\in [0,T]$. The dynamics of $Q^i=(Q_t^i)_{t\in[0,T]}$  is described by, for $t\in(0,T]$,
	\begin{align}\label{dq}
		d Q_t^i=\left[\alpha\left(\overline{Q}_t^{(n)}-Q_t^i\right)+\beta \tau_t^i\right]dt+\sigma  d W_t^i+\sigma_0dW_t^0, \quad Q_0^i=q_0, 
	\end{align}
where $\overline{Q}_t^{(n)}:=\frac{1}{n}\sum_{i=1}^nQ_t^i$ represents the average quality of data among $n$ sellers at time $t\in[0,T]$, $\alpha>0$ is the convergence rate to the average quality in the market and the parameter $\beta>0$ measures the effect of the adjustment. BM $W^i$ represents the stochastic factor that affects the data quality processed by the seller $i$ from the process of data cleaning, data transmission, etc; while BM $W^0$ is the common noise in the data market.  In fact, data quality is determined by multiple attributes that interact in complex ways. Therefore, rather than examining individual quality dimensions in isolation, we model data quality as an integrated whole, which can be quantified through a quality score. A phenomenon known in the study of \cite{AM22} as external data is that data from different sellers may be relevant. In addition, data usually need to be integrated with other data sources to achieve greater value (composability in \cite{BW24}). An excessive deviation from the market average undermines interoperability, giving sellers a strong incentive to align their quality with the market mean. These are captured in our model through the mean-reversion term, which reflects the inherent pressure of the market toward quality convergence.

Let $\mathbb{U}^{\rm br}=\mathbb{U}^{\rm bu}$ be the set of $\mathbb{F}^0$-progressively measurable real-valued processes $x=(x_t)_{t\in[0,T]}$ satisfying $\|x\|_{2,T}\leq C_T$ for $C_T>0$ depending on $T$ only.	In the Stackelberg game, the buyer acts as the leader, the broker as the sub-leader, and the sellers as the followers. Their objective functions are denoted by $J^L(\cdot)$, $J^S(\cdot)$, and $J^F(\cdot)$, respectively.  We instantiate the volume of data sold by every seller is positively correlated with the quality he/she contributes with the correlation coefficient $a >0$. The net profit of seller $i$ is the payment from the broker minus the cost, which is given by, for $p^{\rm br}\in\mathbb{U}^{\rm br}$ and $\boldsymbol{\tau}=(\tau^1,\ldots,\tau^n)\in(\mathbb{U}^{{\rm ad}})^n$,
\begin{align}\label{eq:JiF}
		J_i^F(p^{\rm br}, \boldsymbol{\tau})=\mathbb{E}\Bigg[\int_0^T\bigg(\underbrace{a p_t^{\rm br}Q_t^i}_{\text{Payment from broker}}-\underbrace{\Big(\frac{b }{2}(Q_t^i)^2+c  (\tau_t^i)^2\Big)}_{\text{Adjustment cost}}\bigg)dt\Bigg],  
	\end{align}
where $b>0$ is the marginal cost coefficient adjusted per unit quality; while $c>0$ is the cost coefficient for adjustment effort. This objective function incorporates quadratic cost terms in line with the principle of increasing marginal cost. This functional form has been adopted in both static and dynamic models of data pricing (\cite{BL24} and \cite{AX21}).
	
The broker's average profit is calculated by subtracting the seller's remuneration and manufacturing costs from the buyer's payment, for $(p^{\rm bu},p^{\rm br})\in\mathbb{U}^{\rm bu}\times\mathbb{U}^{\rm br}$,
	\begin{align}\label{eq:JSpbupbr}
		J^S(p^{{\rm bu}}, p^{{\rm br}}, \boldsymbol{\tau})=\mathbb{E}\Bigg[\int_0^T\Big[\underbrace{p_t^{{\rm bu}}\frac{\nu}{n}\sum_{i=1}^nQ_t^i}_{\text{Payment from buyer}}-\underbrace{\frac{1}{n}\sum_{i=1}^np_t^{{\rm br}}a Q_t^i}_{\text{Payment to sellers}}-\underbrace{\Big(\frac{\kappa}{n}\sum_{i=1}^n( Q_t^i)^2+\frac{1}{2}(p_t^{{\rm br}})^2\Big)}_{\text{Manufacturing cost}}\Big]dt\Bigg],    
	\end{align}    
where $\nu>0$ is the quality-quantity correlation coefficient which measures how the data product quality affects the sales. $\kappa$ is the manufacture cost parameter for the broker.  The term $\frac{\kappa}{n}\sum_{i=1}^n( Q_t^i)^2$ represents the average manufacturing cost, and the term $\frac{1}{2}(p_t^{\rm br})^2$ is the costs associated with frequent price adjustments for sellers (c.f. \cite{R82}). 
	
The average profit of the buyer from the market transactions is the difference between the utility gained from the data product and the payment to the broker, for $p^{{\rm bu}}\in \mathbb{U}^{\rm bu}$,
	\begin{align}\label{eq:JLpbu}
		J^L(p^{{\rm bu}})=\mathbb{E}\Bigg[\int_0^T\bigg(\underbrace{-\frac{(p_t^{{\rm bu}})^2}{2}+\frac{\lambda}{n}\sum_{i=1}^nQ_t^i-\frac{\rho}{n}\sum_{i=1}^n( Q_t^i)^2}_{\text{Utility}}-\underbrace{p_t^{{\rm bu}}\frac{\nu}{n}\sum_{i=1}^n Q_t^i}_{\text{Payment to broker}}\bigg)dt\Bigg],
	\end{align} 
where the term $\frac{(p_t^{\rm bu})^2}{2}$ represents the quadratic cost of price adjustment over time (c.f. \cite{R82}).  The utility of product quality to buyers is formulated as a quadratic function of the form $U(x):= \lambda x-\rho x^2$ with the slope parameters $\lambda>0$ and $\rho>0$. The utility that buyers obtain from data products mainly depends on their overall quality level. The overall quality level can be measured by the average quality of all data sources. The 3rd term indicates the buyer's aversion to the unevenness of quality (\cite{CL15}) as it will break the composability of data, and will further lead to a deterioration in the performance of data products. The buyer purchases a data product, whose quality is aggregated from all the sellers' data (\cite{BL24}). Ultimately, the buyer cares about the overall quality and stability of the integrated data products, rather than the contribution of individual data sources or the amount of sellers. The general structure of the market is shown in Figure \ref{fig:structure}. 
	
	\begin{figure}[htbp]
		\centering
		\begin{tikzpicture}[
			rolebox/.style={
				draw, rounded corners=8pt, 
				minimum height=1.5cm,
				align=center, font=\sffamily\bfseries,
				drop shadow={shadow xshift=2pt, shadow yshift=2pt}
			},
			buyer/.style={rolebox, fill=blue!10, thick, text=myblue},
			broker/.style={rolebox, fill=orange!5, thick, text=myorange},
			seller/.style={rolebox, fill=green!5, thick, text=mygreen},
			obj/.style={ rectangle, text centered, draw=none}
			flowtext/.style={midway, font=\tiny\sffamily},
			dotnode/.style={inner sep=0pt, minimum size=0pt}]
			\node(buyer)[buyer] at (0,0)  {Buyer \\[-2pt] \footnotesize(Leader)};
			\node (broker) [broker ,below of=buyer, yshift=-2.5cm] {{\footnotesize \makecell{Broker\\(Sub-leader)}}};
			\node (seller2) [seller, below of=broker, yshift=-2.5cm] { \scriptsize{Seller 2} \\[-2pt] \scriptsize(Follower)};
			\node (sellern) [seller, right of =seller2, xshift=1cm] {\scriptsize{Seller n} \\[-2pt] \scriptsize(Follower)};
			\node (seller1) [seller,left of =seller2, xshift=-0.5cm] {\scriptsize{Seller 1} \\[-2pt] \scriptsize(Follower)};
			\node (objective1) [obj, right of=buyer, xshift=5.5cm] {{\footnotesize $	J^L(p^{{\rm bu}})=\mathbb{E}\Bigg[\int_0^T\Big(\underbrace{-\frac{(p_t^{{\rm bu}})^2}{2}+\frac{\lambda}{n}\sum_{i=1}^nQ_t^i-\frac{\rho}{n}\sum_{i=1}^n( Q_t^i)^2}_{\text{Utility}}-\underbrace{p_t^{{\rm bu}}\nu\frac{1}{n}\sum_{i=1}^n Q_t^i}_{\text{Payment to broker}}\Big)dt\Bigg]$}};
			\node (objective2) [obj, right of=broker, xshift=5.9cm]{{\footnotesize $J^S(p^{{\rm bu}}, p^{{\rm br}}, \boldsymbol{\tau})=\mathbb{E}\Bigg[\int_0^T\Big(\underbrace{p_t^{{\rm bu}}\frac{\nu}{n}\sum_{i=1}^nQ_t^i}_{\text{\makecell{Payment \\from buyers}}}-\underbrace{\frac{1}{n}\sum_{i=1}^np_t^{{\rm br}}a Q_t^i}_{\text{\makecell{Payment \\to sellers}}}-\underbrace{\Big(\frac{\kappa}{n}\sum_{i=1}^n( Q_t^i)^2+\frac{(p_t^{{\rm br}})^2}{2}\Big)}_{\text{\makecell{Manufacturing\\ cost}}}\Big)dt\Bigg]$}};
			\node (objective3) [obj, right of=sellern, xshift=4.5cm]{{\footnotesize $J_i^F(p^{{\rm br}}, \boldsymbol{\tau})=\mathbb{E}\Bigg[\int_0^T\Big(\underbrace{p_t^{{\rm br}}a Q_t^i}_{\text{\makecell{Payment\\ from broker}}}-\underbrace{\big(\frac{b}{2}(Q_t^i)^2+c  (\tau_t^i)^2\big)}_{\text{Adjustment cost}}\Big)dt\Bigg]$}};
			\node[dotnode, right=of seller2, xshift=-1cm ] (dots) {$\cdots$};
			\draw[arrow] (buyer) -- node[fill=white, midway, swap] {{\scriptsize \makecell{\textbf{Propose data}\\ \textbf{product price $p^{{\rm bu}}$}}}}(broker);    
			\draw [arrow] (broker) -- (seller1);
			\draw [arrow] (broker) --  node[fill=white, midway, swap] {{\scriptsize \makecell{\textbf{Propose data}\\\textbf{price $p^{{\rm br}}$}}}}(seller2);
			\draw [arrow] (broker) -- (sellern);
			\draw[dashed, line width=1pt] ([xshift=-1.55cm, yshift=0.3cm]buyer.north west) rectangle ([xshift=1cm, yshift=-0.08cm]objective1.south east);
			\draw[dashed, line width=1pt] ([xshift=-1.35cm, yshift=0.4cm]broker.north west) rectangle ([xshift=0.1cm, yshift=-0.1cm]objective2.south east);
			\draw[dashed, line width=1pt] ([xshift=-0.1cm, yshift=0.35cm]seller1.north west) rectangle ([xshift=1.35cm, yshift=-0.3cm]objective3.south east);
		\end{tikzpicture}
		\caption{The data market structure for three parties.}
		\label{fig:structure}
	\end{figure}
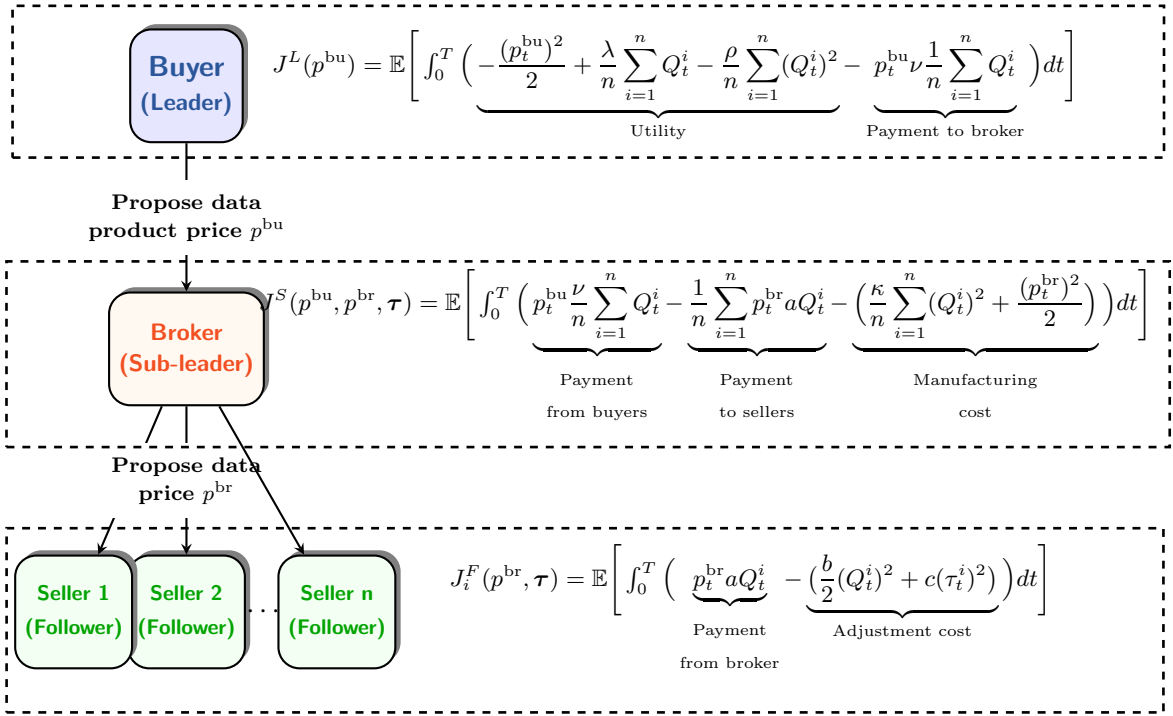
The aim is to seek the optimal strategy of adjustment and pricing for all three parties to maximize their own net profit. To do it, we provide the definition of the Stackelberg equilibrium below:
\begin{definition}[$(\epsilon_1, \epsilon_2, \epsilon_3)$-Stackelberg Equilibrium]
Let $\boldsymbol{\tau}^*=(\tau^{*,1}, \ldots, \tau^{*,n})$ and $\boldsymbol{\tau}^{*,-i}=(\tau^{*,1},$ $ \ldots, \tau^{*, i-1}, \tau^i, \ldots, \tau^{*,n})$. A vector of admissible strategy $(p^{{\rm bu},*}, p^{{\rm br},*}, \boldsymbol{\tau}^*)$ is called an $(\epsilon_1, \epsilon_2, \epsilon_3)$-Stackelberg equilibrium if it holds that
\begin{itemize}
\item[{\rm(i)}] $\boldsymbol{\tau}^*(p^{\rm br})$ constitutes an $\epsilon_1$-Nash equilibrium for a given $p^{\rm br}\in\mathbb{U}^{\rm br}$, i.e., $J_i^F(p^{\rm br}, \boldsymbol{\tau}^*(p^{\rm br}))+\epsilon_1 \geq \sup_{\tau^i\in \mathbb{U}^{\rm ad}} J_i^F(p^{\rm br}, \boldsymbol{\tau}^{*,-i}(p^{\rm br}), \tau^i)$ for all $i=1,\ldots,n$. In other words, given a price process $p^{{\rm br}}$ from the sub-leader, if all followers except $i$ employ a set of control laws $\boldsymbol{\tau}^{*,-i}$, any deviation from $\tau^{*,i}$ can yield a profit increase of at most $\epsilon_1$.

\item[{\rm(ii)}] for a given $p^{\rm bu}\in\mathbb{U}^{\rm bu}$, $J^S(p^{\rm bu}, p^{{\rm br},*}(p^{\rm bu}), \boldsymbol{\tau}^*(p^{{\rm br},*}))+\epsilon_2 \geq \sup_{p^{\rm br}\in\mathbb{U}^{\rm br}}J^S(p^{\rm bu}, p^{\rm br}, \boldsymbol{\tau}^*(p^{\rm br}))$. In other words, for the sub-leader, given the leader’s decision $p^{{\rm bu}}$,  the sub-leader cannot gain more than $\epsilon_2$ by choosing any alternative strategy $p^{{\rm br}}$ anticipating the followers’ response $\boldsymbol{\tau}^*(p^{{\rm br}})$.

\item[{\rm(iii)}] $J^L(p^{{\rm bu},*})+\epsilon_3\geq  \sup_{p^{\rm bu}\in\mathbb{U}^{\rm bu}}J^L(p^{\rm bu})$. In other words, the leader’s strategy $p^{{\rm bu},*}$ is $\epsilon_3$-optimal, ensuring that the leader’s payoff cannot be increased by more than $\epsilon_3$ through any other choice, taking into account the sub-leader’s and followers’ equilibrium responses.
\end{itemize}
\end{definition}

\section{Mean Field Game Problem for Followers}\label{sec:3}

In this section, we consider the following $n$-player game problem for sellers, given an arbitrary strategy $p^{\rm br} \in \mathbb{U}^{\rm br}$ of the broker, for $i=1,\ldots,n$,
\begin{align}\label{npg}
\begin{cases}
\displaystyle\sup_{\tau^i\in \mathbb{U}^{\rm ad}}J_i^F\left(p^{{\rm br}}, \tau\right)=\sup_{\tau^i\in \mathbb{U}^{\rm ad}}\mathbb{E}\left[\int_0^T(p_t^{{\rm br}}a Q_t^i-\frac{b }{2}(Q_t^i)^2-c  (\tau_t^i)^2)dt\right],\\[1em]
\displaystyle~ d Q_t^i=\left(\alpha\left(\overline{Q}_t^{(n)}-Q_t^i\right)+\beta \tau_t^i\right)dt+\sigma d W_t^i+\sigma_0dW_t^0,~~ Q_0^i=q_0.		
\end{cases}
\end{align}
Due to the intricate nature resulting from the coupling of the state-average, it is tractable to solve this problem directly. We turn to study the corresponding MFG problem, and then establish an approximate equilibrium when $n$ is large.

\subsection{Mean field equilibrium}\label{sec:MFE}

We consider the MFG for sellers under an arbitrary strategy $p^{{\rm br}}\in\mathbb{U}^{\rm br}$ of the broker. Let $m = (m_t)_{t \in [0, T]}\in \mathbb{L}_{\mathbb{F}^0}^2([0,T])$ be a continuous process. For a representative seller in the market, the data quality dynamics is given by 
\begin{align}\label{mfg1}
dQ_t^m=\left(\alpha\left(m_t-Q_t^m\right)+\beta\tau_t^m\right)dt+\sigma d W_t+\sigma_0dW_t^0,~~Q_0^m=q_0,
\end{align}
where $W=(W_t)_{t\in[0,T]}$ is a scalar Broanwin motion which is independent of BMs $W^0,W^1,\ldots,W^n$. The optimal control problem for the representative seller is
\begin{align}\label{mfg2}
\sup_{\tau^m\in \mathbb{U}^{\rm ad}}\tilde{J}^F\left(p^{{\rm br}}, m; \tau\right)=\sup_{\tau^m\in \mathbb{U}^{\rm ad}}\mathbb{E}\left[\int_0^T\left(ap_t^{{\rm br}}Q_t^m-\frac{b}{2}(Q_t^m)^2 -c (\tau_t^m)^2\right)dt\right].
\end{align}

The definition of mean field equilibrium (MFE) for sellers is provided below:
\begin{definition}[MFE for sellers]\label{def}
Given $p^{{\rm br}}\in\mathbb{U}^{\rm br}$,  let $\tau^{m,*}\in\mathbb{U}^{\rm ad}$ be the best response to the control problem \eqref{mfg1}-\eqref{mfg2}. Then, $\tau^{m,*}$ is called a MFE if it is the best response to itself in the sense that $m_t=\mathbb{E}[Q_t^{m,*}| \mathcal{F}_t^0]$ for $t\in [0, T]$,  where $Q^{m,*}=(Q_t^{m,*})_{t\in[0,T]}$ satisfies \eqref{mfg1} under $\tau^{m,*}\in\mathbb{U}^{\rm ad}$.
\end{definition}
To establish a MFE for sellers, we first handle the solvability of the control problem \eqref{mfg1}-\eqref{mfg2}, whose solution is documented in the following lemma:   
\begin{lemma}\label{response}
Given $p^{{\rm br}}\in\mathbb{U}^{\rm br}$, let $(Q^{m,*},Y^{m,*},Z^{m,*})=(Q_t^{m,*},Y_t^{m,*},Z_t^{m,*})_{t\in[0,T]}$ depending on $p^{{\rm br}}\in\mathbb{U}^{\rm br}$ satisfy the following FBSDE:
\begin{align}\label{fbsde3.1}
\begin{cases}
    \displaystyle dQ_t^{m,*}=\left(\alpha\left(m_t-Q_t^{m,*}\right)+\frac{\beta^2}{2c}Y_t^{m,*}\right)dt+\sigma dW_t+\sigma_0dW_t^0,~~Q_0^{m,*}=q_0,\\[0.6em]
    \displaystyle dY_t^{m,*}=\left(bQ_t^{m,*}+\alpha Y_t^{m,*}- ap_t^{{\rm br}}\right)dt+Z_t^{m,*}dW_t+Z_t^{m, *,0}dW_t^0,~~Y_{T}^{m,*}=0.
\end{cases}
\end{align}
Define $\tau_t^{m,*}(p^{{\rm br}}):=\frac{\beta}{2c}Y_t^{m,*}$ for $t\in[0,T]$. Then, $\tau^{m,*}(p^{{\rm br}})=(\tau_t^{m,*}(p^{{\rm br}}))_{t\in[0,T]}\in\mathbb{U}^{\rm ad}$ is an optimal strategy for the representative control problem \eqref{mfg1}-\eqref{mfg2}. Moreover, $\tau^{m,*}$ has the following state-feedback form given by 	
$\tau_t^{m,*}(p^{{\rm br}})=\frac{\beta}{2c}(\xi_t Q_t^{m,*}+(\vartheta_t-\xi_t)m_t+\zeta_t)$ for $t\in [0, T]$, where
$(\zeta,\sigma^{\zeta})=(\zeta_t,\sigma_t^{\zeta})_{t\in[0,T]}$ obeys that
\begin{align}\label{fbsde3.2}
\begin{cases}
	\displaystyle dQ_t^{m,*}=\left[\Big(\alpha-\frac{\beta^2}{2c}\xi_t\Big)(m_t-Q_t^{m,*})+\frac{\beta ^2}{2c}(\vartheta_tm_t+\zeta_t)\right]dt+\sigma dW_t+\sigma_0dW_t^0,~Q_0^{m,*}=q_0,\\[0.5em]
	\displaystyle d\zeta_t=\left(\left(\alpha-\frac{\beta^2}{2c}\xi_t\right)\zeta_t- ap_t^{{\rm br}}\right)dt+\sigma_t^{\zeta}dW_t^0,~~\zeta_T=0,
\end{cases}
\end{align}
and the deterministic functions are given by, for $t\in[0,T]$,
\begin{align}\label{xi}
\begin{cases}
\displaystyle \xi_t=\frac{be^{(\delta_1^+-\delta_1^-)(T-t)}-b}{\delta_1^-e^{(\delta_1^+-\delta_1^-)(T-t)}-\delta_1^+},~~\delta_1^{\pm}:=-\alpha\pm\sqrt{\alpha^2+\frac{b\beta^2}{2c}},~~\xi_T=0,\\[0.6em] \displaystyle \vartheta_t=\frac{be^{(\delta_2^+-\delta_2^-)(T-t)}-b}{\delta_2^-e^{(\delta_2^+-\delta_2^-)(T-t)}-\delta_2^+},~ ~\delta_2^{\pm}:=-\frac{\alpha}{2}\pm\sqrt{\frac{\alpha^2}{4}+\frac{b\beta^2}{2c}},~~\vartheta_T=0.
\end{cases}
\end{align}
\end{lemma}

\begin{proof}
We solve problem \eqref{mfg1}-\eqref{mfg2} by using the stochastic maximum principle (SMP). First, we introduce the reduced Hamiltonian given by, for $(t,\omega,q, y,\tau)\in[0,T]\times\Omega\times \mathbb{R}^3$,
\begin{align*}
H^{m}(t, \omega; q, y,\tau)=y(\alpha(m_t(\omega)-q)+\beta\tau)+ap_t^{{\rm br}}(\omega)q-\frac{b}{2}q^2-c\tau^2.    
\end{align*}
For an optimal strategy $\tau^{m,*}=(\tau_t^{m,*})_{t\in[0,T]}\in\mathbb{U}^{\rm ad}$ of control problem \eqref{mfg1}-\eqref{mfg2}, let $(Q^{m,*},Y^{m,*})=(Q^{m,*}_t,Y^{m,*}_t)_{t\in[0,T]}$ be the corresponding state trajectory and the adjoint process, i.e., $(Q^{m,*},Y^{m,*})$ satisfies the following FBSDEs given by 
\begin{align*}
\begin{cases}
    \displaystyle dQ_t^{m,*}=\partial_yH^m(t, \omega; Q_t^{m,*}, Y_t^{m,*}, \tau_t^{m,*})dt+\sigma dW_t+\sigma_0dW_t^0,~~Q_0^{m,*}=p_0,\\[0.6em]
    \displaystyle dY_t^{m,*}=\partial_qH^m(t, \omega; Q_t^{m,*}, Y_t^{m,*}, \tau_t^{m,*})dt+Z_t^{m,*}dW_t+Z_t^{m,*,0}dW_t^0,~~Y_T^{m,*}=0,
\end{cases}
\end{align*}
where $\partial_qH$ (resp. $\partial_yH $) is the partial derivative of $H$ w.r.t. $q$ (resp. $y$). Consequently, we obtain \eqref{fbsde3.1} by substituting $
\tau^{m,*}$ into it. 
To study the well-posedness of FBSDE \eqref{fbsde3.1}, we make an ansatz that $Y_t^{m,*}=\xi_t(Q_t^{m,*}-m_t)+\vartheta_t m_t+\zeta_t$, where $t\mapsto\xi_t$ and $t\mapsto \vartheta_t$ are deterministic functions on $[0, T]$, and $\zeta=(\zeta_t)_{t\in [0,T]}$ is an $\mathbb{F}^0$-adapted process with form $d\zeta=\mu_t^{\zeta}dt+\sigma_t^{\zeta}dW_t^0$. 
For $t\in[0,T]$, denote by $\overline{Q}_t^{m, *}$ (resp. $\overline{Y}_t^{m,*}$) the conditional expectation of $Q_t^{m, *}$ (resp. $Y_t^{m, *}$) given $\mathcal{F}_t^0$, i.e.,
\begin{align}\label{overlineQ}
\overline{Q}_t^{m,*}:=\mathbb{E}[Q_t^{m,*}|\mathcal{F}_t^0], ~~\overline{Y}_t^{m,*}:=\mathbb{E}[Y_t^{m,*}|\mathcal{F}_t^0],~~\forall t\in [0,T] .
\end{align}
Then, it holds that $\overline{Y}_t^{m,*}=\xi_t(\overline{Q}_t^{m,*}-m_t)+\vartheta_t m_t+\zeta_t$ for $t\in[0,T]$.
Applying It\^o's rule to $Y_t^{*, m}$ and using the fact that $m_t=\mathbb{E}[Q_t^{m,*}|\mathcal{F}_t^0]=\overline{Q}_t^{m,*}$ in the equilibrium, we obtain
\begin{align*}
	dY_t^{m,*}
	&=\left(\dot{\xi}_t-\alpha \xi_t+\frac{\beta^2}{2c}\xi_t^2\right)(Q_t^{m,*}-m_t)dt+\left(\dot{\vartheta}_t+\frac{\beta^2}{2c}\vartheta_t^2\right)m_tdt+\left(\mu_t^{\zeta}+\frac{\beta^2}{2c}\vartheta_t\zeta_t\right)dt\\
    &\quad+\sigma \xi_tdW_t+(\sigma_0\vartheta_t+\sigma_t^{\zeta})dW_t^0.
\end{align*}
Comparing with $dY_t^{*, m}=((\alpha \xi_t+b)(Q_t^{*, m}-m_t)- ap_t^{{\rm br}}+bm_t+\alpha \zeta_t+\alpha \vartheta_tm_t)dt+Z_t^{m,*}dW_t+Z_t^{m, *, 0}dW_t^0$, we deduce that $Z_t^{m, *}=\sigma \xi_t$, $Z_t^{m, *, 0}=\sigma_0\vartheta_t+\sigma_t^{\zeta}$, and 
\begin{align}\label{eq:zetalinearBSDE}
	\begin{cases}
		\displaystyle \frac{d\xi_t}{dt}=-\frac{\beta^2}{2c}\xi_t^2+2\alpha \xi_t+b,~~\xi_T=0,\\[0.6em]
		\displaystyle \frac{d{\vartheta}_t}{dt}=-\frac{\beta^2}{2c}\vartheta_t^2+\alpha \vartheta_t+b,~~\vartheta_T=0,\\[0.6em]
		\displaystyle d\zeta_t=\left(\left(\alpha-\frac{\beta^2}{2c}\xi_t\right)\zeta_t- ap_t^{{\rm br}}\right)dt+\sigma_t^{\zeta}dW_t^0,~~\zeta_T=0.
	\end{cases}
\end{align}
Since the Riccati's equations in \eqref{eq:zetalinearBSDE} have unique solutions given by \eqref{xi}, the linear BSDE is solvable. This implies that the coupled FBSDE \eqref{fbsde3.1} admits a unique pair of solutions. We then obtain an equivalent state feedback representation of the optimal control and the corresponding optimal trajectory as \eqref{fbsde3.2}.
\end{proof}	   
    
We then have the following result. 
\begin{lemma}\label{mfe}
Given $p^{{\rm br}}\in\mathbb{U}^{\rm br}$, let $(Q^{m,*},Y^{m,*},Z^{m,*})=(Q_t^{m,*},Y_t^{m,*},Z_t^{m,*})_{t\in[0,T]}$ depending on $p^{{\rm br}}\in\mathbb{U}^{\rm br}$ satisfy FBSDE \eqref{fbsde3.1}.
There exists a unique continuous process $m^*\in \mathbb{L}_{\mathbb{F}^0}^2([0,T])$ satisfying the consistency condition $m_t=\mathbb{E}[Q_t^{m, *}|\mathcal{F}_t^0]$ for $t\in [0, T]$, which also has the closed-form given by
\begin{align}\label{eq:closedformfixedm}
 m_t^* & = \exp\left( \frac{\beta^2}{2c} \int_0^t \vartheta_sds \right)\nonumber\\
 &\quad\times
 \left[ 
 q_0 
 + \frac{\beta^2}{2c} \int_0^t \zeta_s \, \exp\left( -\frac{\beta^2}{2c} \int_0^s \vartheta_u\, du \right) ds 
 + \sigma_0 \int_0^t \exp\left( -\frac{\beta^2}{2c} \int_0^s \vartheta_u \, du \right) \, dW^0_s 
 \right].
\end{align}
Moreover, $\tau^*(p^{{\rm br}}):=\frac{\beta}{2c}Y_t^{m^*,*}$ with   $Y^{m^*,*}$ given by \eqref{fbsde3.1} under $m^*$ is a MFE. It also admits a closed-form $\tau_t^*(p^{{\rm br}})=\frac{\beta^2}{2c}(\xi_tQ_t^{m^*,*}+(\vartheta_t-\xi_t)m_t^*+\zeta_t)$ for $t\in[0,T]$, where $(Q^{m,*}, \zeta)=(Q_t^{m,*},\zeta_t)_{t\in[0,T]}$ is given by \eqref{fbsde3.2}. 
\end{lemma}

\begin{proof}
Note that the consistency condition $m_t=\mathbb{E}[Q_t^{m, *}|\mathcal{F}_t^0]$ becomes that, for $t\in[0,T]$,
\begin{align}\label{eq:consistencemm}
m_t =\overline{Q}_t^{m,*}&=q_0+\sigma_0W_t^{0} + \frac{\beta^2}{2c}\int_0^t\overline{Y}_s^{m, *}ds= q_0+\sigma_0W_t^{0} + \frac{\beta^2}{2c}\int_0^t\left(\xi_s(\overline{Q}_s^{m, *}-m_s)+\vartheta_sm_s+\zeta_s\right)ds\nonumber\\
&= q_0+\sigma_0W_t^{0} + \frac{\beta^2}{2c}\int_0^t\left(\vartheta_sm_s+\zeta_s\right)ds,
\end{align}
where we recall that $\overline{Y}^{m,*}=(\overline{Y}_t^{m,*})_{t\in[0,T]}$, $\overline{Q}^{m,*}=(\overline{Q}_t^{m,*})_{t\in[0,T]}$ and $(\xi,\vartheta,\zeta)=(\xi_t,\vartheta_t,\zeta_t)_{t\in[0,T]}$ are respectively given by \eqref{overlineQ}, \eqref{xi} and \eqref{fbsde3.2}. Then, Eq.~\eqref{eq:consistencemm} satisfied by $m=(m_t)_{t\in[0,T]}$ can be written equivalently as
\begin{align}\label{eq:SDEmtW0}
dm_t=\frac{\beta^2}{2c}\left(\vartheta_tm_t+\zeta_t\right)dt+\sigma_0dW_t^0, ~~m_0=q_0.   
\end{align}
Since $\xi$ is bounded on $[0, T]$ and $p^{{\rm br}}\in \mathbb{U}^{{\rm br}}$, the standard theory of linear BSDE yields that the linear BSDE \eqref{eq:zetalinearBSDE} has a unique solution $(\zeta, \sigma^{\zeta})=(\zeta_t, \sigma_t^{\zeta})_{t\in[0,T]}$ which is square-integrable. Obviously, $\vartheta$ is bounded on $[0,T]$ by virtue of \eqref{eq:zetalinearBSDE}. Therefore, the linear SDE \eqref{eq:SDEmtW0} admits the following explicit strong solution given by \eqref{eq:closedformfixedm}. Moreover, the explicit formula yields $\mathbb{E}[\sup_{t\in [0,T]}|m_t|^2]<\infty$, and hence $m\in \mathbb{L}_{\mathbb{F}^0}^2([0,T])$.
\end{proof}

\subsection{Approximate Nash equilibrium}

In this subsection, we aim to construct an approximate Nash equilibrium for $n$ sellers when $n$ is large. Based on the MFE obtained in subsection \ref{sec:MFE}, we define the following adjustment strategy of sellers by, for $i=1, \ldots, n$,
	\begin{align}\label{nop}
		\tau_t^{*, i}(p^{{\rm br}})=\frac{\beta}{2c}\hat{Y}_t^{*, i},\quad\forall t \in [0, T],
	\end{align}
where $(\hat{Q}^{*,i},\hat{Y}^{*,i},Z^{*,i},Z^{*,0,i})=(\hat{Q}_t^{*,i},\hat{Y}_t^{*,i},Z_t^{*,i},Z_t^{*,0,i})_{t\in[0,T]}$ satisfies the following FBSDE:
	\begin{align}\label{audynamic}
		\begin{cases}
\displaystyle d\hat{Q}_t^{*, i}=\left(\alpha (m_t^*-\hat{Q}_t^{*, i})+\frac{\beta ^2}{2c }\hat{Y}_t^{*, i}\right)dt+\sigma dW_t^i+\sigma_0 dW_t^0,  \quad \hat{Q}_0^{*,i}=q_0,\\[0.8em]
\displaystyle d\hat{Y}_t^{*, i}=\left(b \hat{Q}_t^{*, i}+\alpha  \hat{Y}_t^{*, i}- a p_t^{{\rm br}}\right)dt+Z_t^{*,i}dW_t^i+Z_t^{*,0,i}dW_t^0,\quad  \hat{Y}_T^{*, i}=0
		\end{cases}
\end{align} 
with $m^*=(m_t^{*})_{t\in[0,T]}$ being the fixed point given by \eqref{eq:closedformfixedm} in Lemma \autoref{mfe}. For $i=1,\ldots, n$, the quality dynamics $Q_t^{*,i} =(Q^{*,i})_{t\in[0,T]}$ under the strategy $\tau^{*,i}(p^{\rm br})$ defined by \eqref{nop} is
\begin{align}\label{ndynamic}
dQ_t^{*,i}=\left(\alpha\left(\overline{Q}_t^{*, (n)}-Q_t^{*,i}\right)+\frac{\beta ^2}{2c }\hat{Y}_t^{*, i}\right)dt+\sigma  dW_t^i+\sigma_0 dW_t^0, \quad  Q_0^{*,i}=q_0, 
\end{align}
where $\overline{Q}_t^{*, (n)}:=\frac{1}{n}\sum_{i=1}^n Q_t^{*, i}$ for $t\in[0,T]$. Then, we have
\begin{proposition}[Approximate Nash equilibrium for $n$ sellers]\label{ANE}
Given an arbitrary strategy $p^{\rm br} \in \mathbb{U}^{\rm br}$ of the broker, consider the $n$-player game problem \eqref{npg} for sellers. Then, the vector of strategy $\boldsymbol{\tau}^*(p^{{\rm br}})=(\tau^{*,1}(p^{{\rm br}}), \ldots, \tau^{*,n}(p^{{\rm br}}))$ with $\tau^{*,i}(p^{\rm br})$ being defined by \eqref{nop} is an $\epsilon_1=O(n^{-1})$-Nash equilibrium, i.e., for all $i=1,\ldots,n$, 
\begin{align*}
J_i^F\left(p^{{\rm br}},\boldsymbol{\tau}^*(p^{{\rm br}})\right)+\epsilon_1\geq \sup_{\tau^i\in \mathbb{U}^{\rm ad}}J_i^F\left(p^{{\rm br}},\boldsymbol{\tau}^{*,-i}(p^{{\rm br}}), \tau^i\right).    
\end{align*}
\end{proposition}

Before proving Proposition \ref{ANE}, we need the following auxiliary result. Its proof is deferred to Appendix \ref{sec:appendix} for readability.

\begin{lemma}\label{lemma3.1}
Let $ p^{{\rm br}}\in\mathbb{U}^{\rm br}$. For $i=1,\ldots,n$, recall that $Q^{*,i}=(Q_t^{*,i})_{t\in[0,T]}$ is given by \eqref{ndynamic} and $m^*=(m_t^*)_{t\in[0,T]}\in\mathbb{L}_{\mathbb{F}^0}([0,T])$ is the fixed point given by \eqref{eq:closedformfixedm} in Lemma~\ref{mfe}. Then, there exists a constant $C>0$ independent of $(i,n)$ such that $\mathbb{E}[\sup_{t\in [0,T]}|Q_t^{*,i}|^2]\leq C$. Moreover, it holds that  
 \begin{align*}
\mathbb{E}\left[\sup_{t\in [0,T]}\left|\overline{Q}_t^{*,(n)}-m_t^*\right|^2\right]=O\left(n^{-1}\right).
\end{align*}
\end{lemma}
    
Now, we are in a position to prove Proposition \ref{ANE}.
\begin{proof}[Proof of Proposition \ref{ANE}]
By employing the inequality $2xy\leq \frac{1}{\varepsilon} x^2+\varepsilon y^2$ for any $\varepsilon >0$, we have
\begin{align}\label{eq}
J_i^F(p^{{\rm br}},\boldsymbol{\tau}^{*,-i}(p^{{\rm br}}), \tau^i)
\leq \varepsilon\mathbb{E}\left[\int_0^T|p_t^{{\rm br}}|^2dt\right]+\left(\frac{1}{4\varepsilon}-\frac{b  }{2}\right)\mathbb{E}\left[\int_0^T|\check{Q}_t^i|^2dt\right]-c  \mathbb{E}\left[\int_0^T|\tau_t^i|^2dt\right],
\end{align}
where, for $i=1,\ldots,n$, the process $\check{Q}^i=(\check{Q}_t^i)_{t\in[0,T]}$ satisfies that
\begin{equation}\label{eq:checkQ}
\left\{
\begin{aligned}
   d \check{Q}_t^i &=\left[\alpha \left(\frac{1}{n}\sum_{k=1}^n\check{Q}_t^k-\check{Q}_t^i\right)+\beta  \tau_t^i\right]dt+\sigma dW_t^i+\sigma_0dW_t^0,~~\check{Q}_0^i=q_0, \\[0.6em]
  d \check{Q}_t^j &=\left[\alpha\left(\frac{1}{n}\sum_{k=1}^n\check{Q}_t^k-\check{Q}_t^j\right)+\beta\tau_t^{*,j}\right]dt+\sigma d W_t^j+\sigma_0dW_t^0,~~\check{Q}_0^j=q_0,~~j\neq i.
\end{aligned}
\right.
\end{equation}
Then, by using the standard estimate argument, there exists a constant $C > 0$ depending on $T$ only such that, for $t\in[0,T]$,
\begin{align}\label{Qk}
\begin{cases}
	\displaystyle \mathbb{E}\left[|\check{Q}_t^i|^2\right]\leq C\left\{1+ \int_0^t\mathbb{E}\left[|\check{Q}_s^i|^2\right]ds+\int_0^t\frac{1}{n}\sum_{i=1}^k\mathbb{E}\left[|\check{Q}_s^k|^2\right]ds+\mathbb{E}\left[\int_0^T|\tau_t^i|^2dt\right]\right\},\\[0.6em]
	\displaystyle \mathbb{E}\left[|\check{Q}_t^j|^2\right]\leq C\left\{1+ \int_0^t\mathbb{E}\left[|\check{Q}_s^j|^2\right]ds+\int_0^t\frac{1}{n}\sum_{i=1}^k\mathbb{E}\left[|\check{Q}_s^k|^2\right]ds\right\},~~j\neq i.
\end{cases}
\end{align}
By averaging the above inequalities, we have
\begin{equation*}
	\frac{1}{n}\sum_{k=1}^n\mathbb{E}\left[|\check{Q}_t^k|^2\right]\leq C\left\{1+\frac{1}{n}\mathbb{E}\left[\int_0^t|\tau_s^i|^2ds\right]+\int_0^t\frac{1}{n}\sum_{k=1}^n\mathbb{E}\left[|\check{Q}_s^k|^2\right]ds\right\}.
\end{equation*}
Using the Gronwall’s lemma and the estimate \eqref{Qk}, we deduce $\sup_{t\in[0,T]}\mathbb{E}\left[|\check{Q}_t^i|^2\right]\leq C(1+\Vert\tau^i\Vert_{2,T}^2)$. Substituting this estimate into \eqref{eq}, it yields that 
\begin{align*}
J_i^F\left(p^{{\rm br}},\boldsymbol{\tau}^{*,-i}(p^{{\rm br}}), \tau^i\right)\leq C_1-\left(c-\frac{CT}{4\varepsilon}+\frac{b CT}{2}\right)\left\|\tau^i\right\|_{2,T}^2
\end{align*}
with $C_1:=\varepsilon \Vert p^{{\rm br}}\Vert_{2, T}^2+\frac{CT}{4\varepsilon}-\frac{bCT}{2}$. Hence, 
we may take $\varepsilon>0$ large enough such that $c-\frac{CT}{4\varepsilon}+\frac{bCT}{2}>0$. Then, for $\tau^i\in \mathbb{U}^{\rm ad}$, if there exists a constant $C_2$ such that $J_i^F(p^{{\rm br}}, \boldsymbol{\tau}^{*,-i}, \tau^i)\geq C_2$, consequently $\Vert \tau^i \Vert_{2,T}^2=\mathbb{E}[\int_0^T|\tau_t^i|^2dt]\leq \frac{C_1-C_2}{c-\frac{CT}{4\varepsilon}+\frac{bCT}{2}}=:C_3$. Therefore, it suffices to prove the proposition under $\Vert\tau^i\Vert_{2,T}^2\leq C_3$ for some positive $C_3$ depends on $T>0$ only. To this purpose, let us introduce the following auxiliary control problem given by
\begin{align}\label{auxproblem}
\begin{cases}
\displaystyle \sup_{\tau^i\in \mathbb{U}^{\rm ad}}\tilde{J}_i^F(p^{{\rm br}}, \tau^i)=\sup_{\tau^i\in \mathbb{U}^{\rm ad}}\mathbb{E}\left[\int_0^T\left(ap_t^{{\rm br}}\hat{Q}_t^i-\frac{b}{2}(\hat{Q}_t^i)^2-c(\tau_t^i)^2\right)dt\right],	\\[0.6em]
\displaystyle d \hat{Q}_t^i=\left(\alpha(m_t^*-\hat{Q}_t^i)+\beta\tau_t^i\right)dt+\sigma d W_t^i+\sigma_0dW_t^0,~~ \hat{Q}_0^i=q_0.		
\end{cases}
\end{align}
Following the similar argument used in the proof of Lemma \ref{response}, we obtain the optimal strategy $\hat{\tau}_t^{*,i}=\tau_t^{*,i}$ for $t\in [0, T]$. Then, we have 
\begin{align}\label{main}
&\sup_{\tau^i\in \mathbb{U}^{\rm ad}}J_i^F\left(p^{{\rm br}},\boldsymbol{\tau}^{*,-i}(p^{{\rm br}}), \tau^i\right)-J_i^F\left(p^{{\rm br}},\boldsymbol{\tau}^*(p^{{\rm br}})\right)\leq \sup_{\tau^i\in \mathbb{U}^{\rm ad}}\left( J_i^F(p^{{\rm br}},\boldsymbol{\tau}^{*,-i}(p^{{\rm br}}), \tau^i)-\tilde{J}_i^F(p^{{\rm br}}, \tau^i)\right)\nonumber\\
&\qquad\qquad+\left(\tilde{J}_i^F(p^{{\rm br}}, \hat{\tau}^{*,i}(p^{{\rm br}}))-J_i^F(p^{{\rm br}},\boldsymbol{\tau}^*(p^{{\rm br}}))\right).
\end{align}
Let us consider the 1st term on RHS of \eqref{main}. Employing the Cauchy-Schwarz inequality and the inequality $(x-y)^2-x^2\leq y^2+2|x||y|$, one has
\begin{align}\label{term1}&J_i^F(p^{{\rm br}},\boldsymbol{\tau}^{*,-i}(p^{{\rm br}}), \tau^i)-\tilde{J}_i^F(p^{{\rm br}}, \tau^i)=\mathbb{E}\left[\int_0^T\left(ap_t^{{\rm br}}(\check{Q}_t^i-\hat{Q}_t^i)+\frac{b}{2}(\hat{Q}_t^i)^2-\frac{b}{2}(\check{Q}_t^i)^2\right)dt\right]\nonumber\\
&\quad\leq a\left\{\mathbb{E}\left[\int_0^T|p_t^{{\rm br}}|^2dt\right]\mathbb{E}\left[\int_0^T|\check{Q}_t^i-\hat{Q}_t^i|^2dt\right]\right\}^{\frac{1}{2}}\nonumber\\
&\qquad+\frac{b}{2}\left\{\mathbb{E}\left[\int_0^T|\check{Q}_t^i-\hat{Q}_t^i|^2dt\right]+2\mathbb{E}\left[\int_0^T|\check{Q}_t^i-\hat{Q}_t^i|^2dt\right]^{\frac{1}{2}}\mathbb{E}\left[\int_0^T|\hat{Q}_t^i|^2dt\right]^{\frac{1}{2}}\right\}.
\end{align}
We also have that 
\begin{align}\label{checkhatQ}
&\mathbb{E}\left[|\check{Q}_t^i-\hat{Q}_t^i|^2\right]\leq C\left\{ \int_0^t\mathbb{E}\left[|\check{Q}_s^i-\hat{Q}_s^i|^2\right]ds+\int_0^t\mathbb{E}\left[\left|\frac{1}{n}\sum_{k=1}^n\check{Q}_s^k-m_s^*\right|^2\right]ds\right\}\\
&\quad \leq C\left\{ \int_0^t\mathbb{E}\left[|\check{Q}_s^i-\hat{Q}_s^i|^2\right]ds+\frac{1}{n}\sum_{k=1}^n\int_0^t\mathbb{E}\left[|\check{Q}_s^k-Q_s^{*,k}|^2\right]ds+\int_0^t\mathbb{E}\left[|\overline{Q}_s^{*,(n)}-m_s^*|^2\right]ds\right\}.\nonumber
\end{align}
It follows from \eqref{ndynamic} and \eqref{eq:checkQ} that, for $i=1,\ldots,n$,
\begin{align*}
\begin{cases}
\displaystyle\mathbb{E}\left[|\check{Q}_t^i-Q_t^{*,i}|^2\right]	\leq C\left\{ \int_0^t\mathbb{E}\left[|\check{Q}_s^i-Q_s^{*,i}|^2\right]ds+\frac{1}{n}\sum_{k=1}^n\int_0^t\mathbb{E}\left[|\check{Q}_s^k-Q_s^{*,k}|^2\right]ds+\mathbb{E}\left[\int_0^t|\tau_s^{*,i}-\tau_s^i|^2ds\right]\right\},\\[0.6em]
\displaystyle\mathbb{E}\left[|\check{Q}_t^j-Q_t^{*,j}|^2\right]\leq C\left\{ \int_0^t\mathbb{E}\left[|\check{Q}_s^j-Q_s^{*,j}|^2\right]ds+\frac{1}{n}\sum_{k=1}^n\int_0^t\mathbb{E}\left[|\check{Q}_s^k-Q_s^{*,k}|^2\right]ds\right\},~~j\neq i.
\end{cases}
\end{align*}
By averaging the above inequalities, it results in that
\begin{align*}
&\frac{1}{n}\sum_{k=1}^n\mathbb{E}\left[|\check{Q}_t^k-Q_t^{*,k}|^2\right]\leq C\left\{\frac{1}{n}\mathbb{E}\left[\int_0^T|\tau_s^{*,i}(p^{{\rm br}})-\tau_s^i|^2ds\right]+\int_0^t\frac{1}{n}\sum_{k=1}^n\mathbb{E}\left[|\check{Q}_s^k-Q_s^{*,k}|^2\right]ds\right\}\nonumber\\
&\qquad\leq C\left\{\frac{1}{n}\mathbb{E}\left[\int_0^T(|\tau_s^{*,i}(p^{{\rm br}})|^2+|\tau_s^i|^2)ds\right]+\int_0^t\frac{1}{n}\sum_{k=1}^n\mathbb{E}\left[|\check{Q}_s^k-Q_s^{*,k}|^2\right]ds\right\}.
\end{align*}
The Gronwall's lemma yields that
$\frac{1}{n}\sum_{k=1}^n\sup_{t\in [0,T]}\mathbb{E}\left[\left|\check{Q}_t^k-Q_t^{*,k}\right|^2\right]=O\left(n^{-1}\right).$
This implies from \eqref{checkhatQ} that $\sup_{t\in [0,T]}\mathbb{E}\left[|\check{Q}_t^i-\hat{Q}_t^i|^2\right]=O(n^{-1})$. On the other hand, the condition $\Vert \tau^i\Vert_{2, T}^2\leq C_3$ leads to $\sup_{t\in [0,T]}\mathbb{E}[|\hat{Q}_t^i|^2]\leq C$ for some constant $C>0$ depending on $T$ only.  Thus, we obtain from \eqref{term1} that
\begin{align}\label{result1}
J_i^F\left(p^{{\rm br}},\boldsymbol{\tau}^{*,-i}(p^{{\rm br}}), \tau^i\right)-\tilde{J}_i^F\left(p^{{\rm br}}, \tau^i\right)\leq O\left(n^{-\frac{1}{2}}\right).
\end{align}
For the 2nd term on the RHS of \eqref{main}, we have
\begin{align*}
&\tilde{J}_i^F\left(p^{{\rm br}}, \hat{\tau}^{*,i}(p^{{\rm br}})\right)-J_i^F\left(p^{{\rm br}},\boldsymbol{\tau}^*(p^{{\rm br}})\right)=\mathbb{E}\left[\int_0^T\left(ap_t^{{\rm br}}(\hat{Q}_t^{*,i}-Q_t^{*,i})+\frac{b}{2}(Q_t^{*,i})^2-\frac{b}{2}(\hat{Q}_t^{*,i})^2\right)dt\right]\\
&\quad\leq\left\{a\mathbb{E}\left[\int_0^T|p_t^{{\rm br}}|^2dt\right]\mathbb{E}\left[\int_0^T|\hat{Q}_t^{*,i}-Q_t^{*,i}|^2dt\right]\right\}^{\frac{1}{2}}\\
&\qquad+\frac{b}{2}\left\{\mathbb{E}\left[\int_0^T|\hat{Q}_t^{*,i}-Q_t^{*,i}|^2dt\right]+2\mathbb{E}\left[\int_0^T|\hat{Q}_t^{*,i}-Q_t^{*,i}|^2dt\right]^{\frac{1}{2}}\mathbb{E}\left[\int_0^T|Q_t^{*,i}|^2dt\right]^{\frac{1}{2}}\right\}.
\end{align*}
From \eqref{audynamic} and \eqref{ndynamic}, it follows that 
\begin{align}\label{eqdifference}
\mathbb{E}\left[\left|\hat{Q}_t^{*,i}-Q_t^{*,i}\right|^2\right]\leq C\left\{\int_0^t\mathbb{E}\left[\left|\overline{Q}_s^{*,(n)}-m_s^*\right|^2\right]ds+\int_0^t\mathbb{E}\left[\left|\hat{Q}_s^{*,i}-Q_s^{*,i}\right|^2\right]ds\right\}.
\end{align}
Under Lemma \ref{lemma3.1}, we can deduce
$\sup_{t\in [0,T]}\mathbb{E}\left[|\hat{Q}_t^{*,i}-Q_t^{*,i}|^2\right]=O(n^{-1})$. It yields that
\begin{align}\label{result2}
\tilde{J}_i^F\left(p^{{\rm br}}, \hat{\tau}^{*,i}(p^{{\rm br}})\right)-J_i^F\left(p^{{\rm br}},\boldsymbol{\tau}^*(p^{{\rm br}})\right)\leq O\left(n^{-\frac{1}{2}}\right).
\end{align}
As a result, the desired result follows from \eqref{result1} and \eqref{result2}.
\end{proof}

\section{Optimal Control Problem for Leaders}\label{sec:4}

This section analyzes the optimal control problem for both the sub-leader (broker) and the leader (buyer) in the hierarchy.

\subsection{Optimal control problem for the sub-leader}	

Given $p^{{\rm bu}}\in\mathbb{U}^{\rm bu}$ and $\boldsymbol{\tau}^*=(\tau^{*,1}, \ldots, \tau^{*,n})\in(\mathbb{U}^{\rm ad})^n$ given by \eqref{nop}, we investigate the optimal control problem confronted by the broker, which is formulated by
\begin{align}\label{brp}
&\sup_{p^{{\rm br}}\in\mathbb{U}^{\rm br}}J^S\left(p^{{\rm bu}}, p^{{\rm br}}, \boldsymbol{\tau}^*\right)\nonumber\\
&\qquad=\sup_{p^{{\rm br}}\in\mathbb{U}^{\rm br}}\mathbb{E}\left[\int_0^T\left(p_t^{{\rm bu}}\frac{\nu}{n}\sum_{i=1}^nQ_t^i-\frac{1}{n}\sum_{i=1}^np_t^{{\rm br}}aQ_t^i-\frac{\kappa}{n}\sum_{i=1}^n( Q_t^i)^2-\frac{(p_t^{{\rm br}})^2}{2}\right)dt\right],
\end{align}
where $2\kappa\geq a^2$ which is used to guarantee the concavity of the objective function w.r.t. $p^{{\rm br}}$. Since we have adopted the approximate Nash equilibrium strategies $\boldsymbol{\tau}^*=(\tau^{*,1}, \ldots, \tau^{*,n})\in(\mathbb{U}^{\rm ad})^n$ implemented by sellers for optimal control problem for leaders, we shall use the simplified notation $Q^i=(Q_t^i)_{t\in[0,T]}$ to represent $Q^{*,i}$ under $\tau^{*,i}$ which satisfies \eqref{ndynamic} for notational convenience, i.e., for $i=1,\ldots,n$,
\begin{align}\label{fbsde4.1}
\begin{cases}
\displaystyle dQ_t^i=\left(\alpha(\overline{Q}_t^{(n)}-Q_t^i)+\frac{\beta^2}{2c}\hat{Y}_t^i\right)dt+\sigma dW_t^i+\sigma_0dW_t^0,~~Q_0^i=q_0,\\[0.6em]
\displaystyle d\hat{Y}_t^i=\left(b\hat{Q}_t^i+\alpha\hat{Y}_t^i- ap_t^{{\rm br}}\right)dt+Z_t^idW_t^i+Z_t^{0,i}dW_t^0,~~\hat{Y}_T^i=0,\\[0.6em]
\displaystyle d\hat{Q}_t^i=\left(\alpha\left(\overline{\hat{Q}_t^i}-\hat{Q}_t^i\right)+\frac{\beta^2}{2c}\hat{Y}_t^i\right)dt+\sigma dW_t^i+\sigma_0dW_t^0,~~\hat{Q}_0^i=q_0.
\end{cases}
\end{align}  

The primary difficulties in solving this problem stem from the curse of dimensionality and complex coupling structures. With $n$ followers, the system involves $3n$ state variables, making direct analytical or numerical methods infeasible for large $n$. Moreover, the problem exhibits multiple layers of coupling, includes mean-field coupling in the dynamics of $Q^i$, aggregate coupling in the objective, and centralize control coupling through the control variable $p^{{\rm br}}$ influence on all state dynamics. This complex structure leads to a high-dimensional, fully coupled FBSDE system, which is analytically intractable.  Instead of tackling this original problem \eqref{brp}-\eqref{fbsde4.1} directly, we introduce the following auxiliary problem with the corresponding objective functional given by
\begin{align}\label{aux1}
\sup_{p^{{\rm br}}\in\mathbb{U}^{\rm br}}\tilde{J}^S\left(p^{{\rm bu}}, p^{{\rm br}}\right):=\sup_{p^{{\rm br}}\in\mathbb{U}^{\rm br}}\mathbb{E}\left[\int_0^T\left(\nu p_t^{{\rm bu}} Q_t-ap_t^{{\rm br}}Q_t-\kappa Q_t^2-\frac{1}{2}(p_t^{{\rm br}})^2\right)dt\right],
\end{align}
subject to 
\begin{align}\label{eq:stateSta0}
\begin{cases}
    \displaystyle dQ_t=\left(\alpha\left(\overline{Q}_t-Q_t\right)+\frac{\beta^2}{2c}Y_t\right)dt+\sigma dW_t+\sigma_0dW_t^0,~~Q_0=q_0,\\[0.8em]
    \displaystyle dY_t=\left(bQ_t+\alpha Y_t- ap_t^{{\rm br}}\right)dt+Z_tdW_t+Z_t^0dW_t^0,~~Y_T=0.
\end{cases}
\end{align}
The following result provides the solution to the optimal control problem \eqref{aux1}-\eqref{eq:stateSta0}:
\begin{lemma}\label{pbr}
Given $p^{{\rm bu}}\in\mathbb{U}^{\rm bu}$, there exists an optimal pricing strategy $p^{{\rm br},*}\in\mathbb{U}^{\rm br}$ to problem \eqref{aux1}-\eqref{eq:stateSta0}, which is given by
\begin{align}\label{asymptotic1}
    p_t^{{\rm br},*}=-a\left(\overline{Q}_t^*+\overline{V}_t^*\right),~~ \forall t\in [0, T], 
\end{align}
where the process $(Q^*,Y^*,Z^*,Z^{*,0},U^*,\Phi^*,\Phi^{*,0},V^*)=(Q_t^*,Y_t^*,Z_t^*,Z_t^{*,0},U_t^*,\Phi_t^*,\Phi_t^{*,0},V_t^*)_{t\in[0,T]}$ satisfies the following FBSDE:
\begin{align}\label{fbsde4.2}
\begin{cases}
\displaystyle dQ_t^*=\left(\alpha\left(\overline{Q}_t^*-Q_t^*\right)+\frac{\beta^2}{2c}Y_t^*\right)dt+\sigma dW_t+\sigma_0dW_t^0,~~Q_0^*=q_0,\\[0.8em]
\displaystyle dY_t^*=\left(bQ_t^*+\alpha Y_t^*+a^2\left(\overline{Q}_t^*+\overline{V}_t^*\right)\right)dt+Z_t^*dW_t+Z_t^{*,0}dW_t^0,~~Y_T^*=0,\\[0.8em]
\displaystyle dU_t^*=\left(2\kappa Q_t^*-bV_t^*+\alpha U_t^*-\nu p_t^{{\rm bu}}-a^2\left(\overline{Q}_t^*+\overline{V}_t^*\right)-\alpha \overline{U}_t^*\right)dt+\Phi_t^* dW_t\\[0.8em]
\displaystyle\qquad\quad+\Phi_t^{*, 0}dW_t^0,~~U_T^*=0,\\[0.8em]
\displaystyle dV_t^*=\left(-\alpha V_t^*-\frac{\beta^2}{2c}U_t^*\right)dt,~~V_0^*=0.
\end{cases}
\end{align}
\end{lemma}
    
\begin{proof}	
Let $p^{{\rm br},*}\in\mathbb{U}^{\rm br}$ be the optimal control of problem \eqref{aux1}-\eqref{eq:stateSta0} and $(Y^*, Z^*,Z^{*,0},Q^*)=(Y_t^*, Z_t^*,Z_t^{*,0}, Q_t^*)_{t\in [0, T]}$ be the corresponding optimal state satisfying the FBSDE \eqref{fbsde4.2}. For any $p^{{\rm br}}\in\mathbb{U}^{\rm br}$ and $\varepsilon>0$, denote by $p^{{\rm br}, \varepsilon}:=p^{{\rm br},*}+\varepsilon \delta p^{{\rm br}}$ with $\delta p^{{\rm br}}:=p^{{\rm br}}-p^{{\rm br},*}$. Let $( Y^{\varepsilon}, Z^{\varepsilon},Z^{\varepsilon,0},Q^{\varepsilon})$ be the state trajectories corresponding to the perturbed strategy $p^{{\rm br}, \varepsilon}\in\mathbb{U}^{\rm br}$. Set $\delta Y:=\frac{1}{\varepsilon}(Y^{\varepsilon}-Y^*)$, $\delta Z:=\frac{1}{\varepsilon}(Z^{\varepsilon}-Z^*)$, $\delta Z^0:=\frac{1}{\varepsilon}(Z^{\varepsilon,0}-Z^{*,0})$ and $\delta Q:=\frac{1}{\varepsilon}(Q^{\varepsilon}-Q^*)$. Then, we have 
\begin{align*}
\begin{cases}
\displaystyle d(\delta Q)_t=\left(\alpha \overline{(\delta Q)}_t-\alpha(\delta Q)_t+\frac{\beta^2}{2c}(\delta Y)_t\right)dt,~~(\delta Q)_0=0,\\[0.8em]
\displaystyle d(\delta Y)_t=\left(b(\delta Q)_t+\alpha (\delta Y)_t-a(\delta p^{{\rm br}})\right)dt+(\delta Z)_tdW_t+(\delta Z^0)_tdW_t^0,~~(\delta Y)_T=0.\\
\end{cases}
\end{align*}
By applying It\^o's formula to $U_t^*(\delta Q)_t$ from $t=0$ to $t=T$ to have that
\begin{align}\label{ito1}
0&=\mathbb{E}[U_T^*(\delta Q)_T-U_0^*(\delta Q)_0]=\mathbb{E}\Bigg[\int_0^T\bigg(U_t^*\left(\alpha \overline{(\delta Q)}_t-\alpha (\delta Q)_t+\frac{\beta^2}{2c}(\delta Y)_t\right)+(\delta Q)_t\big(-\nu p_t^{{\rm bu}}+\alpha U_t^*\nonumber\\
&\qquad\quad-bV_t^*+2\kappa Q_t^*-a^2(\overline{Q}_t^*+\overline{V}_t^*)-\alpha \overline{U}_t^*\big)\bigg)dt\Bigg].
\end{align}
Similarly, we obtain
\begin{align}\label{ito2}
0
&=\mathbb{E}\left[\int_0^T\left(V_t^*\left(b(\delta Q)_t+\alpha (\delta Y)_t-a\delta p_t^{{\rm br}}\right)+(\delta Y)_t\left(-\alpha V_t^*-\frac{\beta^2}{2c} U_t^*\right)\right)dt\right].
\end{align}
As a consequence, it holds that 
{\small\begin{align*}
&\tilde{J}^S\left(p^{{\rm bu}}, p^{{\rm br}, \varepsilon}\right)-\tilde{J}^S\left(p^{{\rm bu}}, p^{{\rm br},*}\right)\nonumber\\
&=\mathbb{E}\left[\int_0^T\bigg(\nu p_t^{{\rm bu}}(Q_t^*+\varepsilon (\delta Q)_t)-a(p_t^{{\rm br},*}+\varepsilon \delta p_t^{{\rm br}})(Q_t^*+\varepsilon (\delta Q)_t)-\kappa(Q_t^*+\varepsilon (\delta Q)_t)^2-\frac{1}{2}(p_t^{{\rm br},*}+\varepsilon \delta p_t^{{\rm br}})^2\bigg)dt\right]\nonumber\\
&\quad-\mathbb{E}\left[\int_0^T\bigg(\nu p_t^{{\rm bu}}Q_t^*-ap_t^{{\rm br},*}Q_t^*-\kappa(Q_t^*)^2-\frac{1}{2}(p_t^{{\rm br},*})^2\bigg)dt\right]=:-\varepsilon X_1-\varepsilon^2 X_2,
\end{align*}}where, the expectations $X_1$ and $X_2$ are defined by
\begin{align}\label{x1-2}
X_1&:=\mathbb{E}\left[\int_0^T\left(-\nu p_t^{{\rm bu}}(\delta Q)_t+a(\delta p_t^{{\rm br}}Q_t^*+p_t^{{\rm br},*}(\delta Q)_t)+2\kappa (\delta Q)_t Q_t^*+p^{{\rm br},*}\delta p_t^{{\rm br}}\right)dt\right],\nonumber\\
X_2&:=\mathbb{E}\left[\int_0^T\left( \kappa\left((\delta Q)_t\right)^2+\frac{1}{2}(\delta p_t^{{\rm br}})^2+a\delta p_t^{{\rm br}}(\delta Q)_t\right)dt\right].
\end{align}
By recalling \eqref{aux1}, we have that $\tilde{J}^S(p^{{\rm bu}}, p^{{\rm br}})$ is concave in $p^{{\rm br}}$ under the condition $ 2\kappa\geq a^2$. This yields that $X_2\geq 0$. In fact, for any nonnegative $\lambda_1,\lambda_2$ satisfying $\lambda_1+\lambda_2=1$, it follows that %
\begin{align*}
0&\geq \lambda_1\tilde{J}^S\left(p^{{\rm bu}}, p^{{\rm br}, \varepsilon}\right)+\lambda_2\tilde{J}^S\left(p^{{\rm bu}}, p^{{\rm br},*}\right)-\tilde{J}^S\left(p^{{\rm bu}}, \lambda_1p^{br, \varepsilon}+\lambda_2p^{{\rm br},*}, \tau\right)\\
&=-\lambda_1\lambda_2\mathbb{E}\left[\int_0^T\left(\kappa((\delta Q)_t)^2+\frac{1}{2}(\delta p^{{\rm br}})^2+\delta p_t^{{\rm br}}a(\delta Q)_t\right)dt\right]=-\lambda_1\lambda_2X_2.
\end{align*}
Due to the optimality of $p^{{\rm br},*}$, we have $\tilde{J}^S(p^{{\rm bu}}, p^{{\rm br}, \varepsilon})-\tilde{J}^S(p^{{\rm bu}}, p^{{\rm br},*})\leq 0$. Thus, it follows from \eqref{ito1}, \eqref{ito2} and \eqref{x1-2} that
\begin{align*}
X_1&=\mathbb{E}\left[\int_0^T\left(a(\delta p_t^{{\rm br}}+(\delta Q)_t)(p^{{\rm br},*}+a\left( \overline{Q}_t^*+ \overline{V}_t^*)\right)+\alpha \left((\delta Q)_t\overline{U}_t^*-\overline{(\delta Q)}_tU_t^*\right)\right)dt\right]\\
&=\int_0^T\mathbb{E}\left[\mathbb{E}\left[\left(a(\delta p_t^{{\rm br}}+(\delta Q)_t)(p^{{\rm br},*}+a( \overline{Q}_t^*+ \overline{V}_t^*))+\alpha ((\delta Q)_t\overline{U}_t^*- \overline{(\delta Q)}_tU_t^*)\right)\bigg|\mathcal{F}_t^0\right]\right]dt\\
&=\int_0^T\mathbb{E}\left[\left(a(\delta p_t^{{\rm br}}+\overline{(\delta Q)}_t)(p^{{\rm br},*}+a( \overline{Q}_t^*+ \overline{V}_t^*))\right)\right]dt.
\end{align*}
By the arbitrariness of $\varepsilon>0$, we have $X_1=0$. Thus, we deduce that \eqref{asymptotic1} is the optimal solution due to the arbitrariness of $p^{{\rm br}}\in\mathbb{U}^{\rm br}$.
\end{proof}

\subsection{The solvability of FBSDE \eqref{fbsde4.2}}

We now apply the decoupling technique based on the four-step scheme to study the solvability of FBSDE \eqref{fbsde4.2}. To this purpose, taking conditional expectations on both sides of \eqref{fbsde4.2}, we have
    \begin{align}\label{fbsde4.2conditional}
    	\begin{cases}
    		\displaystyle d\overline{Q}_t^*=\frac{\beta^2}{2c}\overline{Y}_t^*dt+\sigma_0dW_t^0, ~~ \overline{Q}_0^*=q_0,\\[0.8em]
    		\displaystyle d\overline{Y}_t^*=\left((b+a^2)\overline{Q}_t^*+a^2\overline{V}_t^*+\alpha \overline{Y}_t^*\right)dt+Z_t^{*,0}dW_t^0,~~ \overline{Y}_T^*=0,\\[0.8em]
    		\displaystyle d\overline{U}_t^*=\left((2\kappa-a^2) \overline{Q}_t^*-b\overline{V}_t^*-\nu p_t^{{\rm bu}}-a^2\overline{V}_t^*\right)dt+\Phi_t^{*, 0}dW_t^0,~~ \overline{U}_T^*=0,\\[0.8em]
    		\displaystyle d\overline{V}_t^*=\left(-\alpha \overline{V}_t^*-\frac{\beta^2}{2c}\overline{U}_t^*\right)dt,~~ \overline{V}_0^*=0.
    	\end{cases}
    \end{align}
Denote by $\boldsymbol{X}:=(\overline{Q}^*, \overline{V}^*)$ and $\boldsymbol{Y}:=(\overline{Y}^*, \overline{U}^*)$. Then, Eq.~\eqref{fbsde4.2conditional} can be read as: 
    \begin{align}\label{overlinex}
    	\begin{cases}
    		\displaystyle d\boldsymbol{X}_t=(A_1\boldsymbol{X}_t+B_1 \boldsymbol{Y}_t)dt+DdW_t^0,~~ \boldsymbol{X}_0=(q_0,0)^{\top},\\[0.6em]
    		\displaystyle d\boldsymbol{Y}_t=(A_2 \boldsymbol{X}_t+B_2\boldsymbol{Y}_t+{\rm C}_t)dt+{\rm D}_tdW_t^0,~~\boldsymbol{Y}_T=(0,0)^{\top},
    	\end{cases}	
    \end{align}
    where the matrices of coefficients are given by
    \begin{align}\label{eq:matrxi0000}
    	A_1&=\begin{pmatrix}
    	0& 0\\
    	0&-\alpha
    \end{pmatrix},~~ B_1=\begin{pmatrix}
    	\frac{\beta^2}{2c}&0\\
    	0&-\frac{\beta^2}{2c}
    \end{pmatrix},~~ A_2=\begin{pmatrix}
    b+a^2&a^2\\
    2\kappa-a^2&-b-a^2
\end{pmatrix},\nonumber\\ 
B_2&=\begin{pmatrix}
\alpha &0\\
0&0
\end{pmatrix},~~ 
{\rm C}_t=\begin{pmatrix}
0\\-\nu p_t^{{\rm bu}}
\end{pmatrix},~~ D=\begin{pmatrix}
    	\sigma_0\\
    	0
    \end{pmatrix},~~ {\rm D}_t=\begin{pmatrix}
    	Z_t^{*,0}\\
    	\Phi_t^{*,0}
    \end{pmatrix}.
\end{align}
Note that both ${\rm C}_t$ and ${\rm D}_t$ for $t\in[0,T]$ are random. Let us consider the following relation between $\boldsymbol{X}$ and $\boldsymbol{Y}$ as $\boldsymbol{Y}_t=-f_t\boldsymbol{X}_t+\ell_t,\quad\forall t\in[0,T], $
where $t\mapsto f_t$ is a deterministic matrix-valued function on $[0,T]$ and $(\ell, \sigma^{\ell})=(\ell_t, \sigma^{\ell})_{t\in [0,T]}$ are $\mathbb{F}^0$-adapted processes satisfying $d\ell_t=\mu_t^{\ell}dt+\sigma_t^{\ell}dW_t^0$ with $\ell_T=0$.  By applying It\^o's rule to $\boldsymbol{Y}_t$ and then comparing coefficients, we have  ${\rm D}_t=f_tD+\sigma_t^{\ell}$ and
\begin{align}\label{ft}
\frac{df_t}{dt}=-f_tA_1+f_tB_1f_t+B_2f_t-A_2,\quad f_T=\boldsymbol{0}_2.
\end{align}
    If the matrix Riccati differential equation (RDE) \eqref{ft} admits a unique solution,  then the following decoupled FBSDE \eqref{ellt} is solvable.  
    \begin{align}\label{ellt}
    \begin{cases}
        \displaystyle d\boldsymbol{X}_t=\left((A_1-B_1f_t)\boldsymbol{X}_t+B_1\ell_t\right)dt+DdW_t^0,~~\boldsymbol{X}_0=(q_0,0)^{\top},\\[0.4em]
    	\displaystyle d\ell_t=\left((f_tB_1+B_2)\ell_t+{\rm C}_t\right)dt+\sigma_t^{\ell}dW_t^0,~~ \ell_T=(0,0)^{\top}.
    \end{cases}
    \end{align}
After obtaining $\overline{Q}^*, \overline{Y}^*, \overline{U}^*$ and $\overline{V}^*$, plugging them into FBSDE \eqref{fbsde4.2}, it becomes a standard linear FBSDE with random coefficients. Observed that the first two equations form a coupled FBSDE with random coefficients: 
\begin{align}\label{eq:BSDE_QY}
\begin{cases}
\displaystyle dQ_t^*=\left(-\alpha Q_t^*+\frac{\beta^2}{2c}Y_t^*\right)dt+\alpha\overline{Q}_t^*dt+\sigma dW_t+\sigma_0dW_t^0,~~Q_0^*=q_0,\\[0.6em]
\displaystyle dY_t^*=\left(bQ_t^*+\alpha Y_t^*\right)dt+a^2\left(\overline{Q}_t^*+\overline{V}_t^*\right)dt+Z_t^*dW_t+Z_t^{*,0}dW_t^0, ~~ Y_T^*=0.
\end{cases}
\end{align}
The FBSDE \eqref{eq:BSDE_QY} is solvable in the sense that there exists a pair of $\mathbb{F}^0$-adapted processes $(\gamma, \sigma^{\gamma})=(\gamma_t, \sigma_t^{\gamma})_{t\in[0,T]}$ satisfying 
\begin{align}\label{gamma}
d\gamma_t=\left(-\frac{\beta^2}{2c}\xi_t+\alpha\right)\gamma_tdt+\left(a^2(\overline{Q}_t^*+\overline{V}_t^*)-\alpha \xi_t\overline{Q}_t^*\right)dt+\sigma_t^{\gamma}dW_t^0, \quad \gamma_T=0,
\end{align}
such that $Y_t^*=\xi_tQ_t^*+\gamma_t$ for $t\in[0,T]$ with $\xi$ given by \eqref{xi}.  Then, substitute $Q^*$ into the 3rd equation in \eqref{fbsde4.2} to have that
    \begin{align}
    	\begin{cases}
    		\displaystyle dU_t^*=\left(-bV_t^*+\alpha U_t^*\big)dt+\big(2\kappa Q_t^*-\nu p_t^{{\rm bu}}-a^2\left(\overline{Q}_t^*+\overline{V}_t^*\right)-\alpha \overline{U}_t^*\right)dt\\[0.6em]
            \displaystyle\qquad\qquad+\Phi_t^* dW_t+\Phi_t^{*, 0}dW_t^0,~~U_T^*=0,\\[0.6em]
    		\displaystyle dV_t^*=\left(-\alpha V_t^*-\frac{\beta^2}{2c}U_t^*\right)dt,~~V_0^*=0.
    	\end{cases}
    \end{align}
    There exists a deterministic function $t\to\phi_t$ and a pair of $\mathbb{F}^0$-adapted processes $(\chi,\sigma^{\chi})=(\chi_t, \sigma_t^{\chi})_{t\in [0,T]}$ such that $U_t^*=\phi_tV_t^*+\chi_t$ for $t\in [0, T]$ satisfying
    \begin{align}
    	d\phi_t&=\left(\frac{\beta^2}{2c} \phi_t^2+2\alpha\phi_t-b\right)dt,~~\phi_T=0.\label{phi}\\ d\chi_t&=\left(\alpha+\frac{\beta^2}{2c}\phi_t\right)\chi_tdt+\left(2\kappa Q_t^*-\nu p_t^{{\rm bu}}-a^2\left(\overline{Q}_t^*+\overline{V}_t^*\right)-\alpha \overline{U}_t^*\right)dt+\sigma_t^{\chi}dW_t^0, ~~ \chi_T=0.\label{phi2}
    \end{align}
Let $C>0$ be a genic constant depending on $T>0$ only, which may be different from line to line.   
By virtue of \eqref{ellt} together with the standard estimates of BSDE and Gronwall's lemma, we have $\sup_{t\in [0,T]}\mathbb{E}[|\ell_t|^2]\leq Ce^{CT}$. Recall $\boldsymbol{X}=(\boldsymbol{X}_t)_{t\in[0,T]}$ given in the 1st equation in \eqref{ellt}. Then, using the Burkholder-Davis-Gundy (BDG) inequality, one has
\begin{align*}
\mathbb{E}\left[\sup_{t\in [0 ,T]}\left|\boldsymbol{X}_t\right|^2\right]\leq C\left(1+\int_0^T\mathbb{E}\left[\sup_{s\in [0, t]}|\boldsymbol{X}_s|^2\right]dt\right).
\end{align*}
It follows from the Gronwall's lemma again that $\mathbb{E}[\sup_{t\in [0,T]}|\boldsymbol{X}_t|^2]\leq Ce^{CT}$. Note that, it can be rewritten as $p_t^{{\rm br},*}=-a\boldsymbol{1}_2^{\top}\boldsymbol{X}_t$ for $t\in[0,T]$. Consequently, $\Vert p^{{\rm br},*} \Vert_{2,T}^2\leq C$, and hence $p^{{\rm br},*}\in \mathbb{U}^{\rm br}$. 
Thus, the existence and uniqueness of solutions to FBSDE \eqref{fbsde4.2} follows from the existence of a unique solution of RDE \eqref{ft}. In other words, once $f$ is known, the existence of a solution to $\ell=(\ell_t)_{t\in[0,T]}$ will immediately follow.
However, since the corresponding RDE \eqref{ft} is not symmetric, it might not admit a unique solution. For asymmetric RDEs, certain sufficient conditions can guarantee the existence of solutions. Based on these considerations, we propose the following lemma, whose proof can be derived from Theorem 4.3 in \cite{MY99} and Lemma 3 in \cite{MB18}.
\begin{lemma}\label{condition1}
Let $\boldsymbol{0}_2$ and $\mathbf{I}_2$ be $2 \times 2$ zero matrix and identity matrix respectively, and recall $A_i,B_i$ for $i=1,2$ given in \eqref{eq:matrxi0000}. If it holds that $\det \left\{ \begin{pmatrix} \mathbf{I}_2 &  \boldsymbol{0}_2\end{pmatrix} e^{M(t-T)} \begin{pmatrix}   \mathbf{I}_2\\ \boldsymbol{0}_2 \end{pmatrix} \right\} > 0$ for $t\in[0,T]$, where $M=	\begin{pmatrix}
A_1 & -B_1\\
-A_2 & B_2 
\end{pmatrix}$, then the RDE \eqref{ft} has a unique solution given by
\begin{align}\label{eq:solutionf}
 f_t = \left[ \begin{pmatrix} \boldsymbol{0}_2 & \mathbf{I}_2 \end{pmatrix} e^{M(t-T)} \begin{pmatrix}\mathbf{I}_2\\ \boldsymbol{0}_2  \end{pmatrix} \right] \left[ \begin{pmatrix}   \mathbf{I}_2&\boldsymbol{0}_2 \end{pmatrix} e^{M(t-T)} \begin{pmatrix} \mathbf{I}_2 \\ \boldsymbol{0}_2 \end{pmatrix} \right]^{-1},~~\forall t\in[0,T].   
\end{align}
\end{lemma}
	
The following theorem establishes the asymptotic optimality of the proposed strategy \eqref{asymptotic1}.
\begin{theorem}[Asymptotically optimal strategy for broker]
Given $p^{{\rm bu}}\in\mathbb{U}^{\rm bu}$, recall $p^{{\rm br},*}$ given by \eqref{asymptotic1} in Lemma~\ref{pbr}. Then, $p^{{\rm br},*}\in\mathbb{U}^{\rm br}$ is asymptotically optimal to problem \eqref{brp}-\eqref{fbsde4.1} in the sense that
there exists $\epsilon_2=O(n^{-\frac{1}{2}})>0$ such that 
\begin{align}\label{main1}
J^S\left(p^{{\rm bu}},p^{{\rm br},*}, \boldsymbol{\tau}^*(p^{{\rm br},*})\right)+\epsilon_2\geq \sup_{p^{{\rm br}}\in\mathbb{U}^{\rm br}}J^S\left(p^{{\rm bu}},p^{{\rm br}}, \boldsymbol{\tau}^*(p^{{\rm br}})\right),
\end{align}
where $\boldsymbol{\tau}^*(p^{{\rm br}})$ for $p^{{\rm br}}\in\mathbb{U}^{\rm br}$ is given by \eqref{nop}.
\end{theorem}

\begin{proof} 
To prove the theorem, let us introduce another auxiliary stochastic control problem as:
\begin{align}\label{eq:auxiliaryproblemthm0}
\sup_{p^{{\rm br}}\in\mathbb{U}^{\rm br}}\tilde{J}_i^S\left(p^{{\rm bu}}, p^{{\rm br}}\right)=\sup_{p^{{\rm br}}\in\mathbb{U}^{\rm br}}\mathbb{E}\left[\int_0^T\left(\nu 
 p_t^{{\rm bu}}\hat{Q}_t^i-p_t^{{\rm br}}a\hat{Q}_t^i-\kappa (\hat{Q}_t^i)^2-\frac{1}{2}(p_t^{{\rm br}})^2\right)dt\right],
\end{align}
subject to 
\begin{align*}
\begin{cases}
\displaystyle d\hat{Q}_t^i=\left(\alpha\left(\overline{\hat{Q}_t^i}-\hat{Q}_t^i\right)+\frac{\beta^2}{2c}\hat{Y}_t^i\right)dt+\sigma dW_t^i+\sigma_0dW_t^0, ~~\hat{Q}_0^i=q_0,\\[0.8em]
 \displaystyle d\hat{Y}_t^i=\left(b\hat{Q}_t^i+\alpha \hat{Y}_t^i- ap_t^{{\rm br}}\right)dt+Z_t^idW_t^i+Z_t^{0,i}dW_t^0,~~\hat{Y}_T^i=0.
\end{cases}
\end{align*}
Using the similar argument in the proof of Lemma~~\ref{pbr}, we can get that the optimal solution of problem \eqref{eq:auxiliaryproblemthm0} is given by $-a(\overline{\hat{Q}}_t^{*,i}+\overline{V}_t^{*,i})$ for $t\in[0,T]$, where the processes 
     \begin{align}\label{QIconditional}
    	\begin{cases}
    		\displaystyle d\overline{\hat{Q}_t^{*,i}}=\frac{\beta^2}{2c}\overline{\hat{Y}_t^{*,i}}dt+\sigma_0dW_t^0, ~~ \overline{\hat{Q}_0^{*,i}}=q_0,\\[0.8em]
    		\displaystyle d\overline{\hat{Y}_t^{*,i}}=\left((b+a^2)\overline{\hat{Q}_t^{*,i}}+a^2\overline{V}_t^{*,i}+\alpha \overline{\hat{Y}_t^{*,i}}\right)dt+Z_t^{*,0,i}dW_t^0,~~ \overline{\hat{Y}_T^{*,i}}=0,\\[0.8em]
    		\displaystyle d\overline{U}_t^{*,i}=\left((2\kappa-a^2) \overline{\hat{Q}_t^{*,i}}-b\overline{V}_t^{*,i}-\nu p_t^{{\rm bu}}-a^2\overline{V}_t^{*,i}\right)dt+\Phi_t^{*,0,i}dW_t^0,~~ \overline{U}_t^{*,i}=0,\\[0.8em]
    		\displaystyle d\overline{V}_t^{*,i}=\left(-\alpha \overline{V}_t^{*,i}-\frac{\beta^2}{2c}\overline{U}_t^{*,i}\right)dt,~~ \overline{V}_0^{*,i}=0.
    	\end{cases}
    \end{align}
Utilizing Lemma \ref{condition1}, the FBSDE system \eqref{QIconditional} is solvable. By the standard uniqueness result for FBSDEs (see e.g., \cite{MY99}), the solutions of the two systems \eqref{fbsde4.2conditional} and \eqref{QIconditional} are almost surely equal. Thus, we have $\sup_{p^{{\rm br}}\in \mathbb{U}^{{\rm br}}}\tilde{J}_i^S(p^{{\rm bu}}, p^{{\rm br}})=\tilde{J}_i^S(p^{{\rm bu}}, p^{{\rm br},*})$.	We have from \eqref{main1} that
{\small\begin{align}\label{main4}
&\sup_{p^{{\rm br}}\in\mathbb{U}^{\rm br}}J^S\left(p^{{\rm bu}},p^{{\rm br}}, \boldsymbol{\tau }^*(p^{{\rm br}})\right)-J^S\left(p^{{\rm bu}},p^{{\rm br},*}, \boldsymbol{\tau}^*(p^{{\rm br},*})\right)\leq \sup_{p^{{\rm br}}\in\mathbb{U}^{\rm br}}\left(J^S(p^{{\rm bu}},p^{{\rm br}}, \boldsymbol{\tau }^*(p^{{\rm br}}))-\frac{1}{n}\sum_{i=1}^n\tilde{J}_i^S(p^{{\rm bu}},p^{{\rm br}})\right)\nonumber \\
&\qquad\qquad\qquad +\left(\frac{1}{n}\sum_{i=1}^n\tilde{J}_i^S(p^{{\rm bu}},p^{{\rm br},*})-J^S(p^{{\rm bu}},p^{{\rm br},*}, \boldsymbol{\tau}^*(p^{{\rm br},*}))\right). 
\end{align}}Consider the 1st term on the RHS of \eqref{main4}, we have
{\small\begin{align*}
&J^S\left(p^{{\rm bu}},p^{{\rm br}}, \boldsymbol{\tau }^*(p^{{\rm br}})\right)-\frac{1}{n}\sum_{i=1}^n\tilde{J}_i^S(p^{{\rm bu}},p^{{\rm br}})\\
&=\mathbb{E}\left[\int_0^T\left(\frac{\nu}{n}\sum_{i=1}^np_t^{{\rm bu}}(Q_t^i-\hat{Q}_t^i)-\frac{a}{n}\sum_{i=1}^np_t^{{\rm br}}(Q_t^i-\hat{Q}_t^i)-\frac{\kappa}{n}\sum_{i=1}^n\left(( Q_t^i)^2-(\hat{Q}_t^i)^2\right)\right)dt\right]\\
&\leq \frac{C}{n}\sum_{i=1}^n\left\{\left(\Vert p^{{\rm bu}}\Vert_{2,T}+\Vert p^{{\rm br}}\Vert_{2,T}+\left( \mathbb{E}\left[\int_0^T|\hat{Q}_t^i|^2dt\right]\right)^{\frac{1}{2}}\right)\left( \mathbb{E}\left[\int_0^T|Q_t^i-\hat{Q}_t^i|^2dt\right]\right)^{\frac{1}{2}} +\mathbb{E}\left[\int_0^T|Q_t^i-\hat{Q}_t^i|^2dt\right]\right\}.
\end{align*}}
Note that, from Lemma \ref{lemma3.1}, it follows that $\mathbb{E}[\int_0^T|Q_t^i-\hat{Q}_t^i|^2dt]=O(n^{-1})$ for any $p^{{\rm br}}\in\mathbb{U}^{\rm br}$. Then, we have, for any $p^{{\rm br}}\in\mathbb{U}^{\rm br}$,
\begin{align}\label{eq:JtildeS10}
J^S\left(p^{{\rm bu}},p^{{\rm br}}, \boldsymbol{\tau }^*(p^{{\rm br}})\right)-\frac{1}{n}\sum_{i=1}^n\tilde{J}_i^S\left(p^{{\rm bu}},p^{{\rm br}}\right)\leq O\left(n^{-\frac{1}{2}}\right).
\end{align}
Regarding the 2nd term on RHS of \eqref{main4}, we have
\begin{align*}
&\frac{1}{n}\sum_{i=1}^n\tilde{J}_i^S(p^{{\rm bu}}, p^{{\rm br},*})-J^S(p^{{\rm bu}},p^{{\rm br},*})\\
&=\mathbb{E}\left[\int_0^T\left(\frac{\nu}{n}\sum_{i=1}^np_t^{{\rm bu}}(\hat{Q}_t^{*,i}-Q_t^{*,i})-\frac{a}{n}\sum_{i=1}^np_t^{{\rm br},*}(\hat{Q}_t^{*,i}-Q_t^{*,i})-\frac{\kappa}{n}\sum_{i=1}^n(( \hat{Q}_t^{*,i})^2-(Q_t^{*,i})^2)\right)dt\right]\\	
	     		&\leq \frac{C}{n}\sum_{i=1}^n\Bigg\{\left(\left\Vert p^{{\rm bu}}\right\Vert_{2,T}+\left\Vert p^{{\rm br},*}\right\Vert_{2,T}+\left( \mathbb{E}\left[\int_0^T|Q_t^{*,i}|^2dt\right]\right)^{\frac{1}{2}}\right)\left( \mathbb{E}\left[\int_0^T|\hat{Q}_t^{*,i}-Q_t^{*,i}|^2dt\right]\right)^{\frac{1}{2}}\\
	     		&\quad +\mathbb{E}\left[\int_0^T|\hat{Q}_t^{*,i}-Q_t^{*, i}|^2dt\right]\Bigg\},
     		\end{align*}
where $Q^{*,i}=(Q_t^{*,i})_{t\in[0,T]}$ satisfies the state dynamics \eqref{ndynamic} under $p_t^{{\rm br},*}$.
By Lemma \ref{lemma3.1} and \eqref{eqdifference}, we also have $\mathbb{E}[\int_0^T|Q_t^{*,i}-\hat{Q}_t^{*,i}|^2dt]=O(n^{-1})$. As a consequence
\begin{align}\label{eq:JtildeS20}
\frac{1}{n}\sum_{i=1}^n\tilde{J}_i^S\left(p^{{\rm bu}},p^{{\rm br},*}\right)-J^S\left(p^{{\rm bu}},p^{{\rm br},*}, \boldsymbol{\tau }^*(p^{{\rm br},*})\right)\leq O\left(n^{-\frac{1}{2}}\right).
\end{align}
Then, the desired result follows from \eqref{eq:JtildeS10} and \eqref{eq:JtildeS20}.
 \end{proof}
 
\subsection{Optimal control problem for the leader}
In this subsection, our goal is to solve the optimal control problem for the leader (i.e., the buyer), with the objective of maximizing the profit function given by
\begin{align}\label{bup}
\sup_{p^{{\rm bu}}\in\mathbb{U}^{\rm bu}}J^L(p^{{\rm bu}})=\sup_{p^{{\rm bu}}\in\mathbb{U}^{\rm bu}}\mathbb{E}\left[\int_0^T\left(-\frac{1}{2}(p_t^{{\rm bu}})^2+\frac{\lambda}{n}\sum_{i=1}^nQ_t^i-\frac{\rho}{n}\sum_{i=1}^n( Q_t^i)^2-p_t^{{\rm bu}}\frac{\nu }{n}\sum_{i=1}^nQ_t^i\right)dt\right],
\end{align}
where $2\rho\geq\nu^2$ which is used to guarantee the concavity of objective function. We next apply the asymptotic pricing strategy of the broker $p^{{\rm br},*}$ given by \eqref{asymptotic1} obtained in the previous subsection in the above optimal control problem. This implies that the control problem \eqref{bup} works with the underlying state processes given by
\begin{align}\label{statebu}
dQ_t^i&=\left(\alpha(\overline{Q}_t^{(n)}-Q_t^i)+\frac{\beta^2}{2c}\hat{Y}_t^i\right)dt+\sigma dW_t^i+\sigma_0dW_t^0,~~ Q_0^i=q_0,\nonumber\\
 d\hat{Q}_t^i&=\left(\alpha(\overline{\hat{Q}_t^i}-\hat{Q}_t^i)+\frac{\beta^2}{2c}\hat{Y}_t^i\right)dt+\sigma dW_t^i+\sigma_0dW_t^0,~~ \hat{Q}_0^i=q_0,\nonumber\\
d\hat{Y}_t^i&=\left(b\hat{Q}_t^i+\alpha \hat{Y}_t^i+a^2(\overline{\hat{Q}_t^i}+\overline{V}_t^i)\right)dt+Z_t^idW_t^i+Z_t^{0,i}dW_t^0,~~ \hat{Y}_T^i=0,\\
d\overline{U}_t^i&=\left(2\kappa \overline{\hat{Q}_t^i}-b\overline{V}_t^i-\nu p_t^{{\rm bu}}-a^2(\overline{\hat{Q}_t^i}+\overline{V}_t^i)\right)dt+\Phi_t^{0,i}dW_t^0,~~ \overline{U}_T^i=0,\nonumber\\
d\overline{V}_t^i&=\left(-\alpha \overline{V}_t^i-\frac{\beta^2}{2c}\overline{U}_t^i\right)dt,~~ \overline{V}_0^i=0,\nonumber
\end{align}
where $\overline{\hat{Q}_t^i}=\mathbb{E}[\hat{Q}_t^i|\mathcal{F}_t^0]$ and similarly for $\overline{\hat{Y}_t^i}$. Note that $\hat{Y}^i$ and $\overline{U}^i$ depend directly on $p^{{\rm br},*}$, while the dynamics of $Q^i$, $\hat{Q}^i$ and $\overline{V}^i$ are consequently also influenced through the coupling of the entire FBSDE system.

To overcome the complexity arising from high dimensionality of the control problem \eqref{bup} in large-scale settings (i.e., as $n \to \infty$), we introduce the following representative agent approximation:
	\begin{equation}\label{aux2}
		\sup_{p^{{\rm bu}}\in\mathbb{U}^{\rm bu}}\tilde{J}^L(p^{{\rm bu}})=\sup_{p^{{\rm bu}}\in\mathbb{U}^{\rm bu}}\mathbb{E}\left[\int_0^T\left(-\frac{1}{2}(p_t^{{\rm bu}})^2+\lambda Q_t-\rho Q_t^2-p_t^{{\rm bu}}\nu Q_t\right)dt\right],
	\end{equation}
	subject to the following state constraints:
	\begin{align}\label{aux2state}
		\begin{cases}
			\displaystyle dQ_t=\left(\alpha\left(\overline{Q}_t-Q_t\right)+\frac{\beta^2}{2c}Y_t\right)dt+\sigma dW_t+\sigma_0dW_t^0,~~Q_0=q_0,\\[0.6em]
			\displaystyle dY_t=\left(bQ_t+\alpha Y_t+a^2(\overline{Q}_t+\overline{V}_t)\right)dt+Z_tdW_t+Z_t^0dW_t^0,~~Y_T=0,\\[0.6em]
			\displaystyle d\overline{U}_t=\left(2\kappa \overline{Q}_t-b\overline{V}_t-\nu p_t^{{\rm bu}}-a^2(\overline{Q}_t+\overline{V}_t)\right)dt+\Phi_t^0dW_t^0,~~ \overline{U}_T=0,\\[0.6em]
			\displaystyle d\overline{V}_t=\left(-\alpha \overline{V}_t-\frac{\beta^2}{2c}\overline{U}_t\right)dt,~~\overline{V}_0=0.
		\end{cases}
	\end{align}
Then, the solvability of the representative problem \eqref{aux2}-\eqref{aux2state} is provided in the following result. The proof of the following lemma is analogous to that of Lemma \ref{pbr}, and hence we omit it.
\begin{lemma}\label{prop4.2}
Let $p_t^{{\rm bu},*}=-\nu (\overline{Q}_t^*+\overline{L}_t^*)$ for $t\in[0,T]$, where the underlying processes are given by
\begin{align}\label{fbsde4.3}
dQ_t^*&=\left(\alpha(\overline{Q}_t^*-Q_t^*)+\frac{\beta^2}{2c}Y_t^*\right)dt+\sigma dW_t+\sigma_0dW_t^0,~~Q_0^*=q_0,\nonumber\\ dY_t^*&=\left(bQ_t^*+\alpha Y_t^*+a^2(\overline{Q}_t^*+\overline{V}_t^*)\right)dt+Z_t^*dW_t+Z_t^{*,0}dW_t^0,~~ Y_T^*=0,\nonumber\\
d\overline{U}_t^*&=\left(2\kappa \overline{Q}_t^*-b\overline{V}_t^*+\nu^2 (\overline{Q}_t^*+\overline{L}_t^*)-a^2(\overline{Q}_t^*+\overline{V}_t^*)\right)dt+\Phi_t^{*,0}dW_t^0,~~ \overline{U}_T^*=0,\nonumber\\
d\overline{V}_t^*&=\left(-\alpha \overline{V}_t^*-\frac{\beta^2}{2c}\overline{U}_t^*\right)dt,~~\overline{V}_T^*=0,\\
dR_t^*&=\left(2\rho Q_t^*+\alpha R_t^*-bK_t^*-2\kappa \overline{L}_t^*-\lambda -\nu^2\overline{Q}_t^*-\alpha \overline{R}_t^*-a^2\overline{K}_t^*+(a^2-\nu^2)\overline{L}_t^*\right)dt\nonumber\\
&\quad+\Psi_t^{*,0}dW_t^0,~~R_T^*=0,\nonumber\\
dK_t^*&=\left(-\frac{\beta^2}{2c}R_t^*-\alpha K_t^*\right)dt,~~K_0^*=0;~~d\overline{L}_t^*=\frac{\beta^2}{2c}\overline{G}_t^*dt,~~\overline{L}_0^*=0,\nonumber\\
d\overline{G}_t^*&=(-a^2\overline{K}_t^*+(b+a^2)\overline{L}_t^*+\alpha \overline{G}_t^*)dt,~~\overline{G}_T^*=0.\nonumber
\end{align}
Then, $p^{{\rm bu},*}=(p_t^{{\rm bu},*})_{t\in[0,T]}\in\mathbb{U}^{\rm br}$ is an optimal control to the representative problem \eqref{aux2}-\eqref{aux2state}.
\end{lemma}

We next handle the solvability of coupled FBSDE \eqref{fbsde4.3}. To do it, we first focus on the equation  after taking the conditional expectation on  $\mathcal{F}_t^0$ on both sides of \eqref{fbsde4.3}, which is given by, for $t\in[0,T]$,
\begin{align}\label{fbsde4.4}
\begin{cases}
\displaystyle d\overline{Q}_t^*=\frac{\beta^2}{2c}\overline{Y}_t^*dt+\sigma_0dW_t^0,~~\overline{Q}_0^*=q_0, \\[0.6em]
\displaystyle d\overline{Y}_t^*=\left((b+a^2)\overline{Q}_t^*+a^2\overline{V}_t^*+\alpha \overline{Y}_t^*\right)dt+Z_t^{*,0}dW_t^0,  ~~\overline{Y}_T^*=0,\\[0.6em]
\displaystyle d\overline{U}_t^*=\left((2\kappa+\nu^2-a^2) \overline{Q}_t^*-b\overline{V}_t^*+\nu^2\overline{L}_t^*-a^2\overline{V}_t^*\right)dt+\Phi_t^{*,0}dW_t^0, \quad \overline{U}_T^*=0,\\[0.6em]
\displaystyle d\overline{V}_t^*=\left(-\alpha \overline{V}_t^*-\frac{\beta^2}{2c}\overline{U}_t^*\right)dt, ~~ \overline{V}_0^*=0,\\[0.6em]
\displaystyle d\overline{R}_t^*=\left((2\rho-\nu^2) \overline{Q}_t^*-(b+a^2)\overline{K}_t^*-\lambda+(a^2-\nu^2-2\kappa)\overline{L}_t^*\right)dt+\Psi_t^{*,0}dW_t^0, ~~\overline{R}_T^*=0, \\[0.6em]
\displaystyle d\overline{K}_t^*=\left(-\frac{\beta^2}{2c}\overline{R}_t^*-\alpha \overline{K}_t^*\right)dt, ~~ \overline{K}_0^*=0;~~d\overline{L}_t^*=\frac{\beta^2}{2c}\overline{G}_t^*dt, ~~\overline{L}_0^*=0.
\end{cases}
\end{align}
Denote by $\mathbb{X}_t=(\overline{Q}_t^*, \overline{V}_t^*, \overline{K}_t^*, \overline{L}_t^*)^{\top}$ and $\mathbb{Y}_t=(\overline{Y}_t^*, \overline{U}_t^*, \overline{R}_t^*, \overline{G}_t^*)^{\top}$ for $t\in[0,T]$. Then, the FBSDE system \eqref{fbsde4.4} can be written as: 
\begin{align}\label{eq:babbXY}
\begin{cases}
\displaystyle d\mathbb{X}_t=(\mathbb{A}_1\mathbb{X}_t+\mathbb{B}_1\mathbb{Y}_t)dt+\mathbb{D}dW_t^0, ~~ \mathbb{X}_0=(q_0,0,0,0)^{\top},\\[0.4em]
\displaystyle d\mathbb{Y}_t=(\mathbb{A}_2\mathbb{X}_t+\mathbb{B}_2\mathbb{Y}_t+\mathbb{C})dt+\widetilde{\mathbb{D}}_tdW_t^0,  \quad \mathbb{Y}_T=(0,0,0,0)^{\top},
\end{cases}	
\end{align}
where the coefficients are given by
\begin{align*}
\mathbb{A}_1& = \diag \left(0,-\alpha,-\alpha,0\right), \quad  \mathbb{B}_1 = \diag \left(\frac{\beta^2}{2c},-\frac{\beta^2}{2c},-\frac{\beta^2}{2c},\frac{\beta^2}{2c}\right), \quad \mathbb{B}_2 = \diag \left(\alpha,0,0,\alpha\right),\\
        		\mathbb{A}_2 &=
        		\begin{pmatrix}
        			b+a^2 &a^2  & 0 & 0 \\
        			2\kappa+\nu^2 - a^2& -b-a^2  & 0 & \nu^2 \\
        			2\rho -\nu^2& 0 & -b-a^2 & -2\kappa+a^2 - \nu^2 \\
        			0 & 0 & -a^2& b+a^2
        		\end{pmatrix}, \\
        		\mathbb{C}& = (0 , 0 , -\lambda , 0)^{\top},\quad
        		\mathbb{D} =(\sigma_0,
        			0,0,0)^{\top}, \quad\widetilde{\mathbb{D}}_t=(Z_t^{*,0},\Phi_t^{*,0},\Psi_t^{*,0},0)^{\top}.
        	\end{align*}
Here, we emphasize that the matrix $\widetilde{\mathbb{D}}_t$ for $t\in[0,T]$ is random.    Consider the form $\mathbb{Y}_t=-g_t\mathbb{X}_t+\psi_t$ with $t\to g_t$ being a deterministic matrix-valued function and $(\psi, \sigma^{\psi})=(\psi_t, \sigma_t^{\psi})_{t\in [0,T]}$ being a pair of $\mathbb{F}^0$-adapted processes satisfying $d\psi_t=\mu_t^{\psi}dt+\sigma_t^{\psi}dW_t^0$.
Applying It\^o's rule to $\mathbb{Y}_t$, and comparing the drift and diffusion terms, we have $\widetilde{\mathbb{D}}_t=-g_t\mathbb{D}+\sigma_t^{\psi}$ for $t\in[0,T]$, while 
\begin{align}\label{rde2}
d{g}_t=\left(-g_t\mathbb{A}_1+g_t\mathbb{B}_1g_t+g_t\mathbb{B}_2-\mathbb{A}_2\right)dt,~~g_T=\boldsymbol{0}_4,
\end{align}
and $d\psi_t=((g_t\mathbb{B}_1+\mathbb{B}_2)\psi_t+\mathbb{C})dt+\sigma_t^{\psi}dW_t^0$ with $\psi_T=(0,0,0,0)^{\top}$. Thus, $(g,\psi)=(g_t,\psi_t)_{t\in[0,T]}$ satisfies a linear FBSDE with random coefficients after plugging $(\mathbb{X}, \mathbb{Y})$ into \eqref{fbsde4.3}. Consequently, the system \eqref{fbsde4.3} can be decomposed into four separated systems of equations in which two of them are undetermined. The 1st and 2nd equations form a coupled FBSDE, which has a unique solution satisfying $Y_t^*=\xi_t Q_t^*+\gamma_t$ with $t\mapsto\xi_t$ and $t\mapsto\gamma_t$ being given by \eqref{xi} and \eqref{gamma}, respectively. 
Plugging $Q_t^*$ and $\overline{L}_t^*$ into the 5th equation, and assuming $R_t^*=\phi_tK_t^*+\eta_t$, we have 
\begin{align*}
d\eta_t=\frac{\beta^2}{2c}\phi_t\eta_tdt+(2\rho Q_t^*-\nu^2\overline{Q}_t^*-2\kappa \overline{L}_t^*-\lambda-a^2\overline{K}_t^*+(a^2-\nu^2)\overline{L}_t^*)dt+\sigma_t^{\eta}dW_t^0,~~ \eta_T=0,
\end{align*}
and $t\to\phi_t$ is given by \eqref{phi}. Recall $p_t^{{\rm bu},*}=-\nu (\overline{Q}_t^*+\overline{L}_t^*)$ for $t\in[0,T]$, which is given in Lemma~\ref{prop4.2}. Hence, $\Vert p^{{\rm bu},*}\Vert_{2, T}^2\leq C\mathbb{E}[\sup_{t\in [0 ,T]}|\mathbb{X}_t|^2]$ with $C>0$ being the constant depending on $T$ only. This implies $p^{{\rm bu},*}\in \mathbb{U}^{{\rm bu}}$.
Similar to Lemma \ref{condition1}, we have the following lemma which provides a condition which can ensure the solvability of $\eqref{rde2}$:
\begin{lemma}\label{condition2}
Let $\mathbb{M}=\begin{pmatrix}
\mathbb{A}_1& -\mathbb{B}_1\\
-\mathbb{A}_2	&\mathbb{B}_2 
\end{pmatrix}$ and $\boldsymbol{0}_4$ (resp. $\mathbf{I}_4$) be $4 \times 4$ zero matrix (resp. identity matrix). Assume $\det \left\{ \begin{pmatrix}\mathbf{I}_4& \boldsymbol{0}_4   \end{pmatrix} e^{\mathbb{M}(t-T)} \begin{pmatrix}  \mathbf{I}_4\\\boldsymbol{0}_4\end{pmatrix} \right\} > 0$ for $t \in [0, T]$. Then, the RDE in \eqref{rde2} has a unique solution given by
\begin{align*}
g_t = \left[ \begin{pmatrix} \boldsymbol{0}_4 & \mathbf{I}_4 \end{pmatrix} e^{\mathbb{M}(t-T)} \begin{pmatrix}   \mathbf{I}_4\\ \boldsymbol{0}_4 \end{pmatrix} \right] \left[ \begin{pmatrix}\mathbf{I}_4 &\boldsymbol{0}_4   \end{pmatrix} e^{\mathbb{M}(t-T)} \begin{pmatrix} \mathbf{I}_4 \\ \boldsymbol{0}_4 \end{pmatrix} \right]^{-1},~~\forall t\in[0,T].
\end{align*}
\end{lemma}
Then, we have
\begin{theorem}[Asymptotically optimal strategy for buyer]\label{thm:pubustar}
The process $p_t^{{\rm bu},*}=(p_t^{{\rm bu},*})_{t\in[0,T]}\in \mathbb{U}^{\rm bu}$ given in Lemma~\ref{prop4.2} is  asymptotically optimal in the sense that
there exists $\epsilon_3=O(n^{-\frac{1}{2}})>0$ such that 
\begin{align*}
J^L\left(p^{{\rm bu},*}\right)+\epsilon_3\geq \sup_{p^{{\rm bu}}\in\mathbb{U}^{\rm bu}}J^L\left(p^{{\rm bu}}\right).
\end{align*}
\end{theorem}

\begin{proof}
Let $C>0$ be a genic constant depending on $T>0$ only, which may be different from line to line. By using the standard estimate for $\mathbb{X}$ and $\mathbb{Y}$ satisfying \eqref{eq:babbXY}, we arrive at $\mathbb{E}[\sup_{t\in [0,T]}(|\mathbb{X}_t|^2+|\mathbb{Y}_t|^2)]\leq C$. 
To prove the theorem, we introduce the following auxiliary control problem given by, for $i=1,\ldots,n$,
\begin{align}\label{aux3}
\sup_{p^{{\rm bu}}\in\mathbb{U}^{\rm bu}}\tilde{J}_i^L(p^{{\rm bu}}):=\sup_{p^{{\rm bu}}\in\mathbb{U}^{\rm bu}}\mathbb{E}\left[\int_0^T\left(-\frac{1}{2}(p_t^{{\rm bu}})^2+\lambda \hat{Q}_t^i-\rho (\hat{Q}_t^i)^2-\nu p_t^{{\rm bu}} \hat{Q}_t^i\right)dt\right],
\end{align}
where the process $\hat{Q}^i=(\hat{Q}_t^i)_{t\in[0,T]}$ satisfies the dynamics \eqref{statebu}.
 Using an argument similar to that in the proof of Lemma~\ref{prop4.2}, together with the uniqueness result of the solution to FBSDE \eqref{fbsde4.4}, we have $\sup_{p^{{\rm bu}}\in\mathbb{U}^{\rm bu}}$ $\tilde{J}_i^L(p^{{\rm bu}})=\tilde{J}_i^L(p^{{\rm bu},*})$. As a result, one has
\begin{align}\label{main5}
&\sup_{p^{{\rm bu}}\in\mathbb{U}^{\rm bu}}J^L\left(p^{{\rm bu}})-J^L(p^{{\rm bu},*}\right)\leq \sup_{p^{{\rm bu}}\in\mathbb{U}^{\rm bu}}\left(J^L\left(p^{{\rm bu}}\right)-\frac{1}{n}\sum_{i=1}^n\tilde{J}_i^L\left(p^{{\rm bu}}\right)\right)\nonumber\\
&\qquad\qquad+\left(\frac{1}{n}\sum_{i=1}^n\sup_{p^{{\rm bu}}\in\mathbb{U}^{\rm bu}}\tilde{J}_i^L\left(p^{{\rm bu},*}\right)-J^L\left(p^{{\rm bu},*}\right)\right).
\end{align}
Consider the 1st term on RHS of the above display \eqref{main5}, we have
{\small\begin{align*}
&J^L\left(p^{{\rm bu}}\right)-\frac{1}{n}\sum_{i=1}^n\tilde{J}_i^L\left(p^{{\rm bu}}\right)=\mathbb{E}\left[\int_0^T\left(p_t^{{\rm bu}}\frac{\nu}{n}\sum_{i=1}^n(\hat{Q}_t^i-Q_t^i)+\frac{\lambda}{n}\sum_{i=1}^n(\hat{Q}_t^i-Q_t^i)-\frac{\rho}{n}\sum_{i=1}^n((Q_t^i)^2-(\hat{Q}_t^i)^2)\right)dt\right]\\
&\quad\leq \frac{C}{n}\sum_{i=1}^n\left\{\left(\Vert p^{{\rm bu}}\Vert_{2,T}+\left( \mathbb{E}\left[\int_0^T|Q_t^i|^2dt\right]\right)^{\frac{1}{2}}\right)\left( \mathbb{E}\left[\int_0^T|Q_t^i-\hat{Q}_t^i|^2dt\right]\right)^{\frac{1}{2}} +\mathbb{E}\left[\int_0^T|Q_t^i-\hat{Q}_t^i|^2dt\right]\right\}.
\end{align*}}Note that $\mathbb{E}[\int_0^T|Q_t^i-\hat{Q}_t^i|^2dt]=O(n^{-1})$ for any $p^{{\rm bu}}\in\mathbb{U}^{\rm bu}$, by using Lemma \ref{lemma3.1}. As a result, it holds that
$J^L\left(p^{{\rm bu}}\right)-\frac{1}{n}\sum_{i=1}^n\tilde{J}_i^L\left(p^{{\rm bu}}\right)\leq O\left(n^{-\frac{1}{2}}\right).$

Regarding the 2nd term on RHS of \eqref{main5},
let $\hat{Q}^{*,i}=(\hat{Q}_t^{*,i})_{t\in[0,T]}$ be the process obtained from $\hat{Q}^i$ in \eqref{statebu} by replacing  $p^{{\rm bu}}$ with its optimal counterpart $p^{{\rm bu},*}$. Then, it can be deduced that
\begin{align*}
&\frac{1}{n}\sum_{i=1}^n\tilde{J}_i^L\left(p^{{\rm bu},*}\right)-J^L\left(p^{{\rm bu},*}\right)\\
&=\mathbb{E}\left[\int_0^T\left(\frac{\nu}{n}\sum_{i=1}^n(p_s^{{\rm bu},*}(\hat{Q}_t^{*,i}-Q_t^{*,i})+\frac{\lambda}{n}\sum_{i=1}^n(\hat{Q}_t^{*,i}-Q_t^{*, i})-\frac{\rho}{n}\sum_{i=1}^n((\hat{Q}_t^{*,i})^2-(Q_t^{*,i})^2)\right)dt\right]\\
&\leq \frac{C}{n}\sum_{i=1}^n\Bigg\{\left(\left\|p^{{\rm bu},*}\right\|_{2,T}+\left( \mathbb{E}\left[\int_0^T|Q_t^{*,i}|^2dt\right]\right)^{\frac{1}{2}}\right)\left( \mathbb{E}\left[\int_0^T|Q_t^{*, i}-\hat{Q}_t^{*,i}|^2dt\right]\right)^{\frac{1}{2}}\\
&\quad +\mathbb{E}\left[\int_0^T|Q_t^{*,i}-\hat{Q}_t^{*,i}|^2dt\right]\Bigg\}.
\end{align*}
Similarly to the proof of the estimate of the first term, we have $\mathbb{E}[\int_0^T|Q_t^{*,i}-\hat{Q}_t^{*,i}|^2dt]=O(n^{-1})$. As a consequence,  we have $\frac{1}{n}\sum_{i=1}^n\tilde{J}_i^L(p^{{\rm bu},*})-J^L(p^{{\rm bu},*})\leq O(n^{-\frac{1}{2}})$. Thus, it results in $\sup_{p^{{\rm bu}}\in\mathbb{U}^{\rm bu}}J^L(p^{{\rm bu}})-J^L(p^{{\rm bu},*}))\leq O(n^{-\frac{1}{2}})$. Therefore, the proof of the theorem is complete.
\end{proof}

\section{Numerical Analysis}\label{sec:numerical}

In this section, we conduct numerical simulations to demonstrate key characteristics of our model. The model parameters used in these simulations are detailed in Table \ref{table:para}.
	\begin{table}[htbp]
		\centering
		\caption{The chosen model parameter values.}
		\label{table:para}
		\begin{tabular}{ccc}
			\toprule
			\textbf{Model Parameters} & \textbf{Financial Meaning} & \textbf{Values} \\
			\midrule
			$\alpha$ & adjustment speed of quality toward average level & $0.12$ \\
			$\beta$ &   adjustment effect & $0.4$ \\
			$\sigma$ & idiosyncratic volatility  & $0.5$ \\
			$\sigma_0 $ & common volatility & $0.2$ \\
			$q_0$ & initial quality value &$-1.0$\\
			$a$ & quantity-quality coefficient of datasets & $0.5$ \\
			$b$ & marginal cost coefficient & $0.4$ \\
			$c$ & adjustment effort cost coefficient &$0.03$\\
			$\kappa $ & manufacture cost parameter& $0.3$ \\
			$\rho $ & marginal utility coefficient& $0.25$ \\
			$\lambda$ &quadratic utility parameter&$0.6$\\
			$\nu$& quantity-quality coefficient of data products & $0.7$\\
			$T$& time horizon & $1.0$\\
			$n$ &the volume of sellers &$30$\\
			\bottomrule
		\end{tabular}
	\end{table}
We note that the parameter values are chosen primarily for illustrative purposes rather than precise calibration, as our goal is to showcase the model's key properties. With the selected parameters, conditions $2\kappa \geq a^2$ and $2\rho \geq \nu^2$ are satisfied.
	\begin{figure}[h]
		\centering
		\begin{subfigure}[b]{0.45\textwidth}
			\centering
			\includegraphics[width=\textwidth]{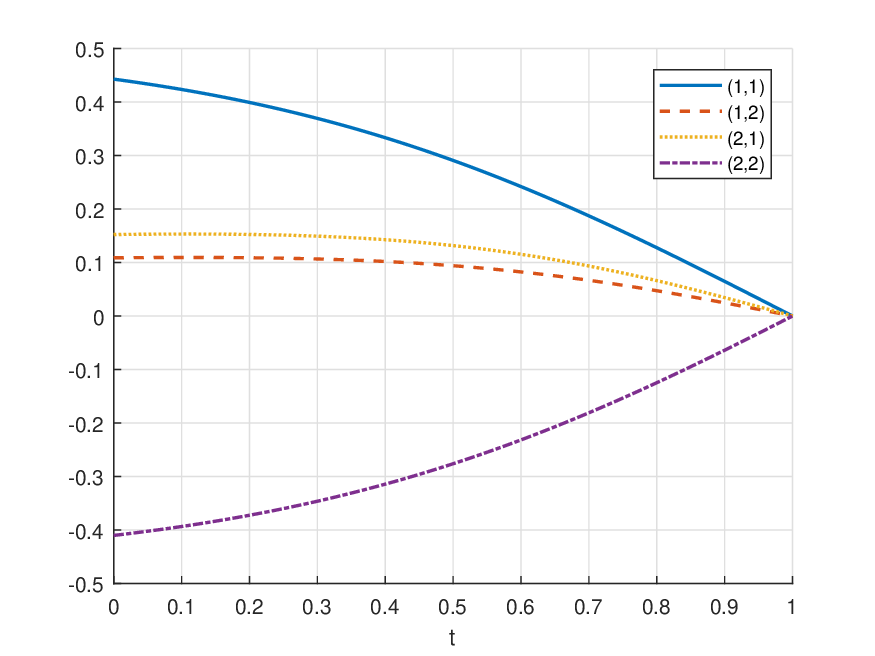}
			\caption{The solutions of RDE \eqref{ft}.}
		\end{subfigure}
		\begin{subfigure}[b]{0.45\textwidth}
			\centering
			\includegraphics[width=\textwidth]{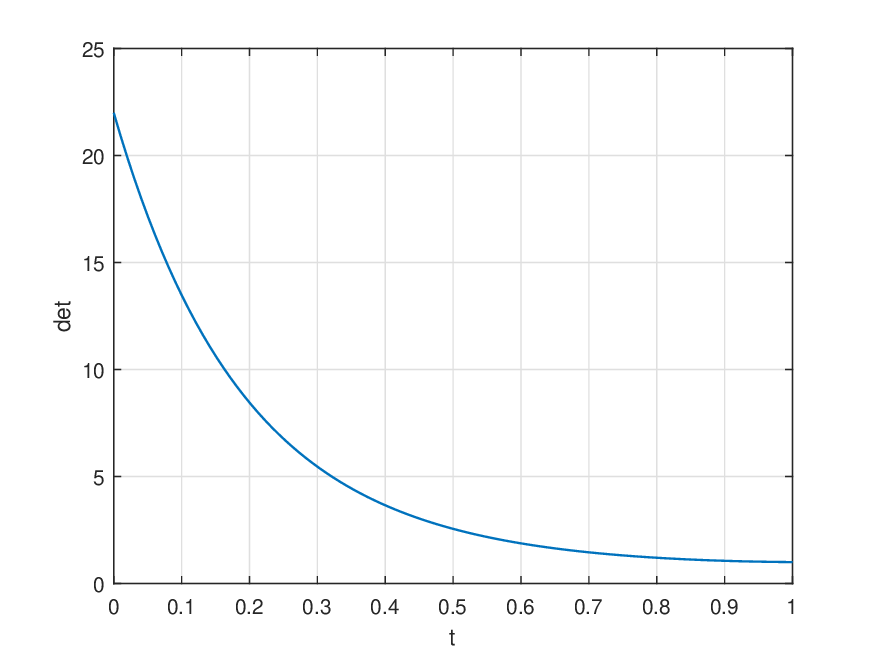}
			\caption{The determinant value in Lemma \ref{condition2}.}
		\end{subfigure}
		\caption{The solvability of Riccati equations.}
		\label{fig:rde}
	\end{figure}
    It can be verified that, with the aforementioned parameters, the sufficient conditions outlined in Lemma \ref{condition1} and \ref{condition2} are satisfied. The solution of the RDE \eqref{ft}, denoted by $f \in \mathbb{R}^{3\times 3}$, is illustrated in Figure \ref{fig:rde}-a and the determinant value in Lemma \ref{condition2} is depicted in Figure \ref{fig:rde}-b.

	\begin{figure}[h]
		\centering
		\begin{subfigure}[b]{0.32\textwidth}
			\centering
			\includegraphics[width=\textwidth]{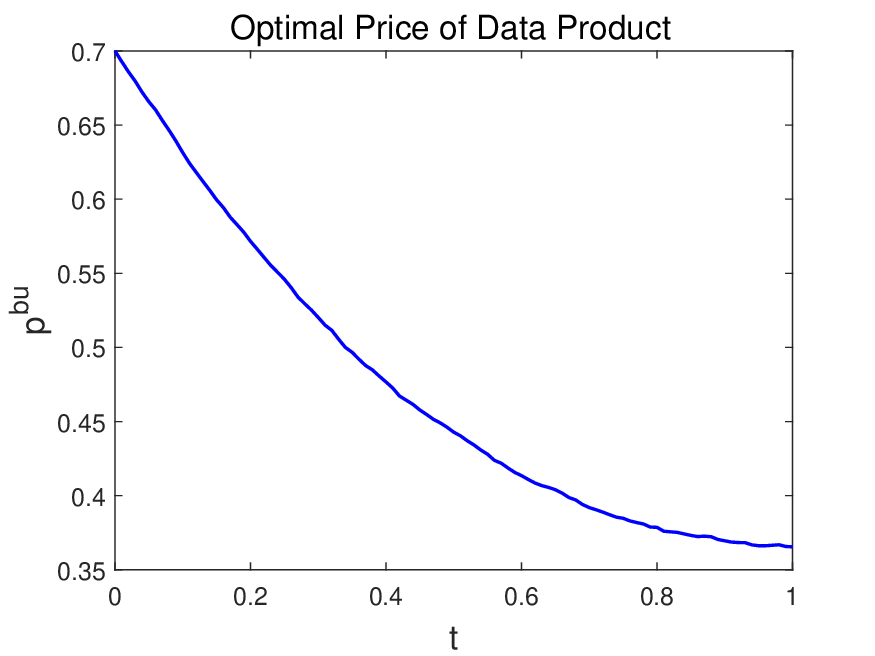}
		\end{subfigure}
		\begin{subfigure}[b]{0.32\textwidth}
			\centering
			\includegraphics[width=\textwidth]{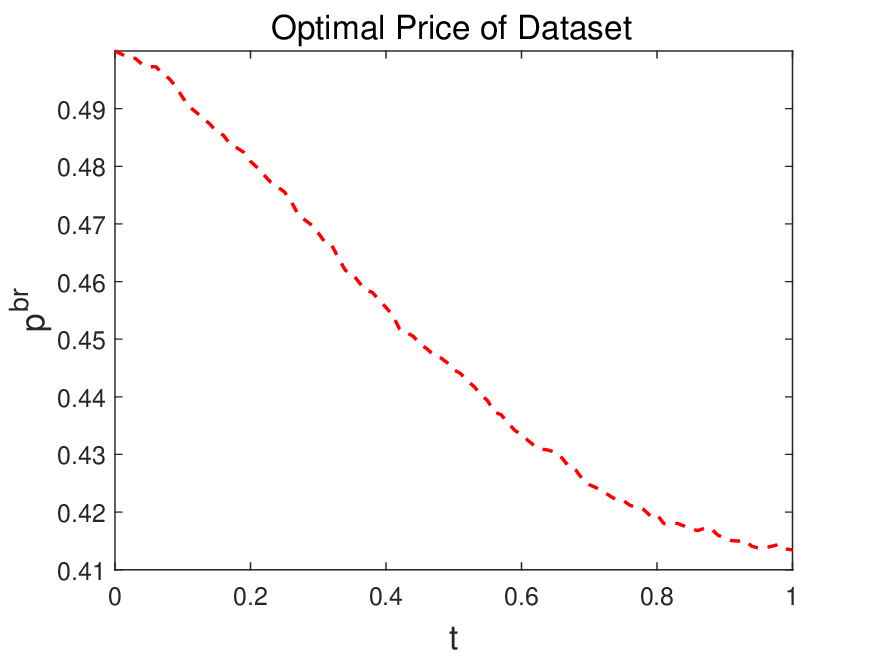}
		\end{subfigure}
		\begin{subfigure}[b]{0.32\textwidth}
			\centering
			\includegraphics[width=\textwidth]{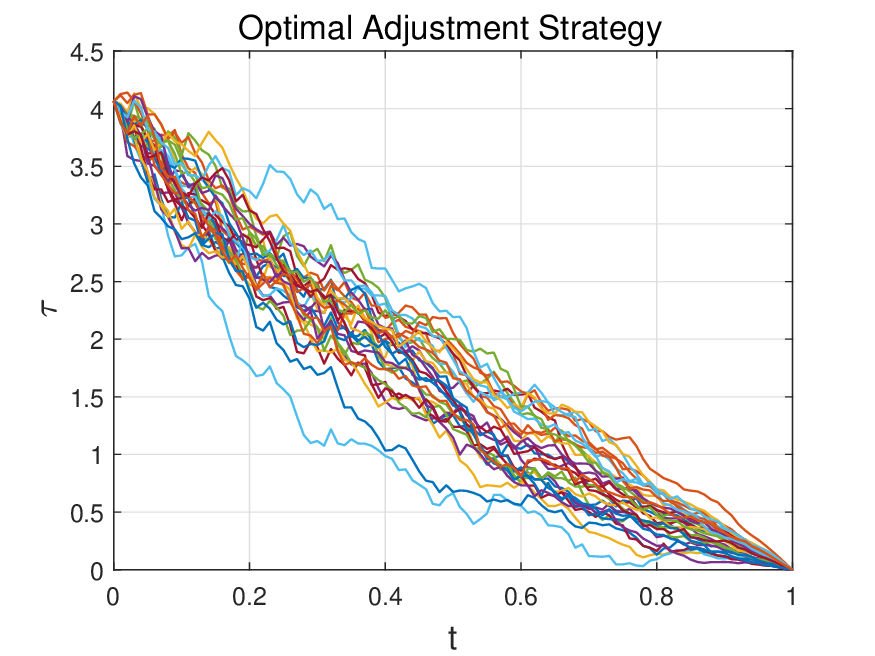}
		\end{subfigure}
		\caption{The trajectories of prices and adjustment strategies. }
		\label{fig:three}
	\end{figure} 
	
	In Figure \ref{fig:three}, we plot simulated price trajectories $p^{{\rm bu}}$, $p^{{\rm br}}$, starting from the initial time (corresponding to $t = 0$), up to the time $T$, together with the adjustment strategy $\boldsymbol{\tau}$ for $n = 30$ sellers.
	Based on the observation of the trajectories of prices and adjustment strategies in Figure \ref{fig:three}, the following features can be observed: all three trajectories display a downward trend.
	Compared to $p^{{\rm bu}}$, the curve of $p^{{\rm br}}$ has experienced a more significant decline. While the trajectories of individual sellers exhibit different levels of volatility—a reflection of their unique, idiosyncratic shocks—they share a co-movement in their overall trend. This common directional tendency is a direct manifestation of the mean-reverting force. It indicates that despite diverse short-term fluctuations, the long-term evolution of all sellers is anchored and pulled towards a common equilibrium level by the same market mechanism.
	
	It can be seen from Figure \ref{fig:three} that, in the early stage, the buyer pays a high price to stimulate the supply of high-quality data. In response to this high offer, the broker also increases his willing-to-pay price to the sellers. In response to the high purchase price, the seller has made great efforts to adjust the data quality.
	As the inherent quality decays, the buyer is essentially facing a diminishing return on quality. To achieve a balance between cost savings and profit maintenance, the buyer takes first to decrease his bid. This signal of price reduction was passed down step by step.  The broker receives the decline in the buyer's price, which lowers the  bid to the seller. As a result, sellers will also reduce their investment in quality. When making efforts decline, the quality will not collapse immediately. The smooth change of the state variable (quality) forces the control variables (price/effort) to also change smoothly. As prices continue to drop, the marginal benefit of further price reduction decreases, while the potential risk of quality loss  remains relatively unchanged. Therefore, the impetus for price reduction weakens and the rate of decline naturally slows down.
	
	Subsequently, we perform a sensitivity analysis on the main parameters of the model.  We select three distinct parameters from the objective functions of three parties and examine the impact of these parameters on the strategy of three parties. 
	\begin{figure}[h]
		\centering
		\begin{subfigure}[b]{0.32\textwidth}
			\centering
			\includegraphics[width=\textwidth]{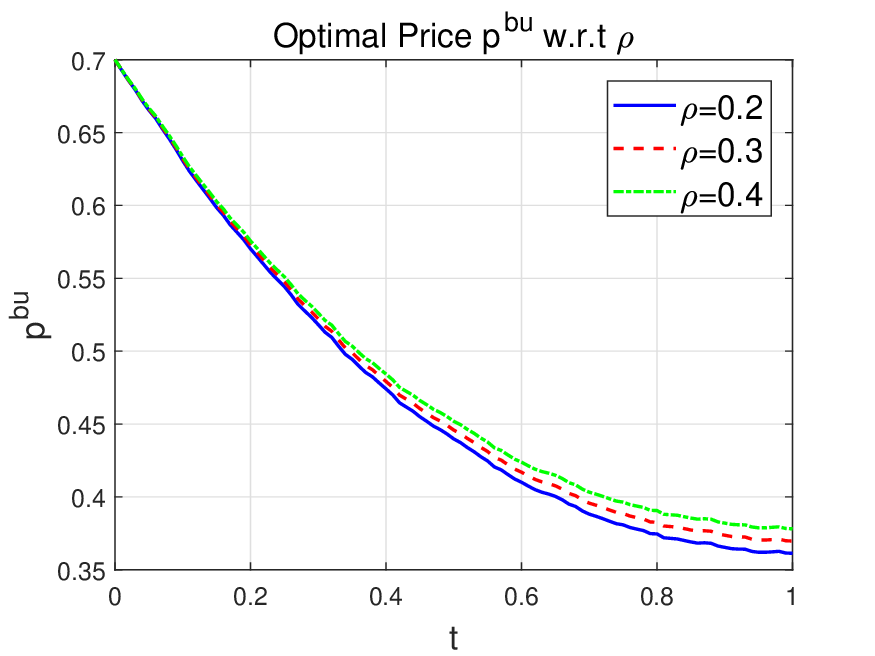}
			\caption{}
		\end{subfigure}
		\begin{subfigure}[b]{0.32\textwidth}
			\centering
			\includegraphics[width=\textwidth]{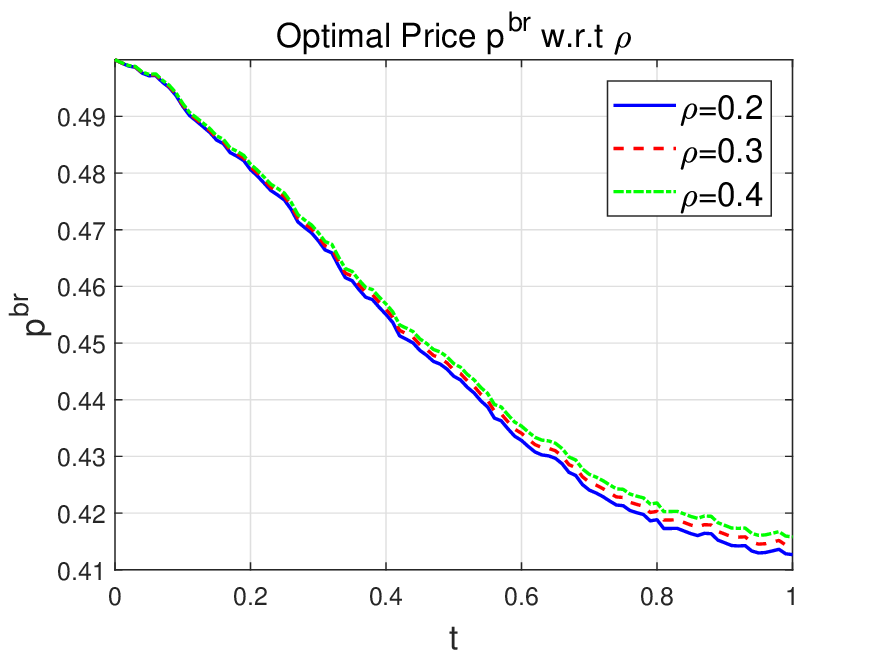}
			\caption{}
		\end{subfigure}
		\begin{subfigure}[b]{0.32\textwidth}
			\centering
			\includegraphics[width=\textwidth]{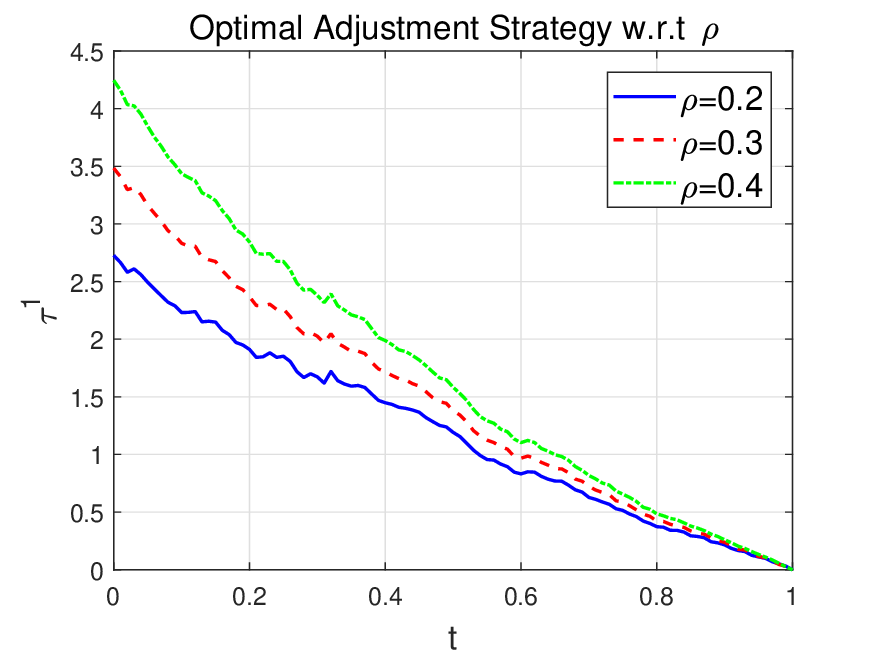}
			\caption{}
		\end{subfigure}
		\caption{The strategies of pricing and adjustment w.r.t  marginal utility coefficient $\rho$.}
		\label{fig:rho}
	\end{figure}
	
Figure \ref{fig:rho} presents the effect of marginal utility coefficient $\rho$ on the strategies of three parties. The graphs were all simulated with other parameters as specified in Table \ref{table:para} while $\rho$ takes different values. The influence of leaders upon sub-leader and followers is highlighted here. We observe that as the value of $\rho$ increases, the first two graphs exhibit a downward trend to varying degrees. In other words, the increase in $\rho$ has raised the prices. Note that an increase in $\rho$ indicates that the buyer's aversion to uneven, unreliable, or unstable data quality has grown. The buyer realizes that if the offer is too low, the sellers will not have sufficient motivation to maintain and improve the data quality, resulting in a low overall quality level and high volatility, which will seriously damage the buyer's utility. At this time, his optimal strategy is to increase the data product price to brokers. This signal propagates down the chain, compelling the broker to correspondingly increase the data set price. A higher price $p^{{\rm br}}$ will bring stronger marginal returns to sellers, prompting them to invest more resources  to enhance and maintain a higher $Q^i$. When  most sellers thereby improve their quality levels, the minimum quality level is raised and the variance of the overall quality will decrease.

	Figure \ref{fig:kappa} depicts how $\kappa$ affect three parties' strategies. $\kappa$ represents the coefficient of the marginal cost of the process by which the broker produces the original data set into a data product. It directly affects the pricing strategy of broker. We can see that as $\kappa$ increases, the price $p^{{\rm br}}$ rises correspondingly. It can be interpreted as for fixed $p^{{\rm bu}}$ set by the buyer, the broker's optimal response is to raise its price $p^{{\rm br}}$ to obtain higher-quality data to cover its increased operational costs. Sellers, responding to the higher $p^{{\rm br}}$, increase their adjustment effort $\tau$, leading to a higher aggregate supply. The buyer, foreseeing this chain of events, recognizes that its initial price will now result in a larger-than-desired market quality. Due to the concavity of its utility function, a higher supply reduces marginal utility. Therefore, to balance the oversupply market and maintain utility, the buyer's optimal  strategy is to lower its price $p^{{\rm bu}}$. The buyer's objective function has no direct relationship with $\kappa$, but it will have an indirect impact through equilibrium. This reflects the typical Stackelberg behavior of leader (buyer) adjusting strategies by expecting the equilibrium results of sub-games.
	
	\begin{figure}[h]
		\centering
		\begin{subfigure}[b]{0.32\textwidth}
			\centering
			\includegraphics[width=\textwidth]{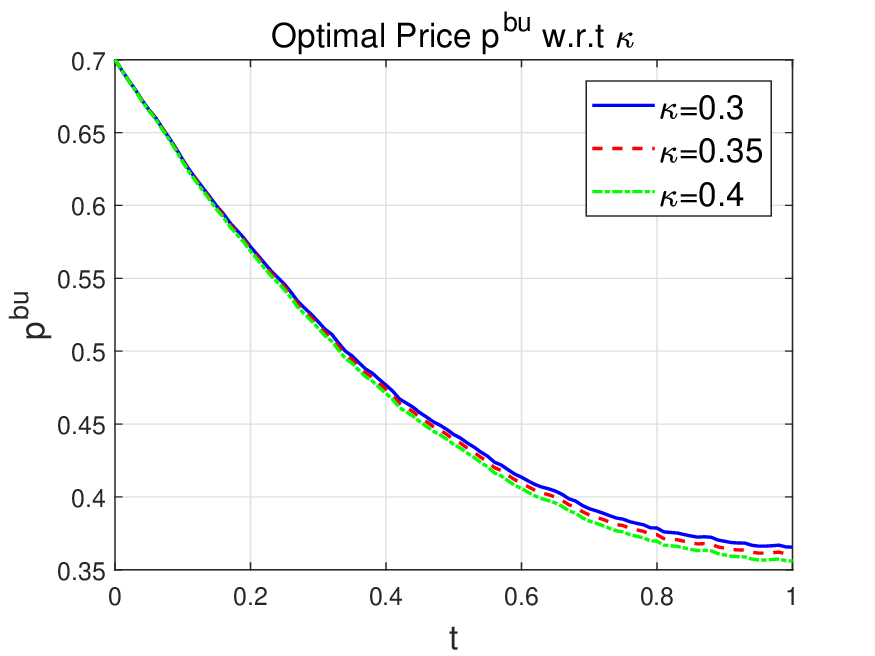}
		\end{subfigure}
		\begin{subfigure}[b]{0.32\textwidth}
			\centering
			\includegraphics[width=\textwidth]{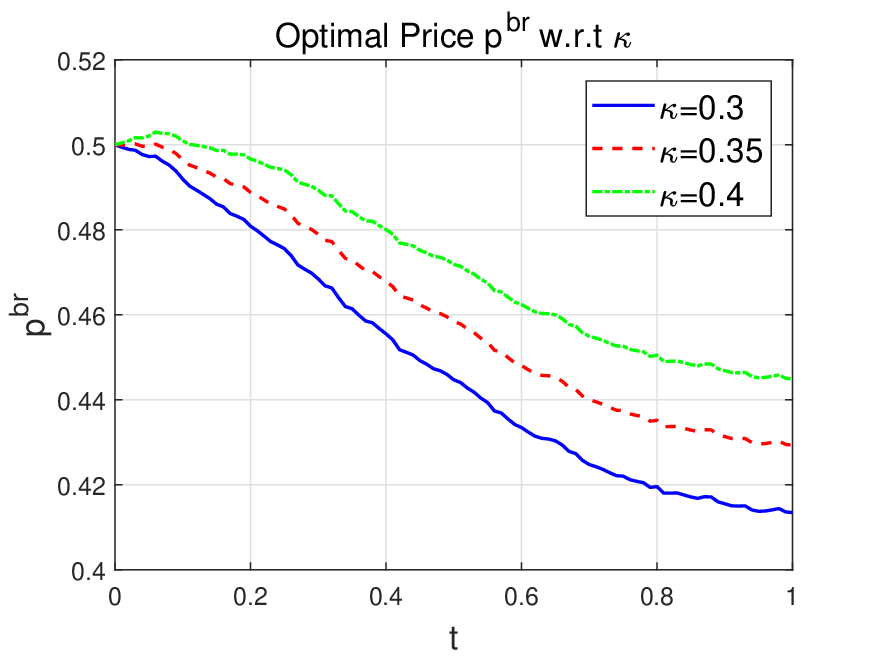}
		\end{subfigure}
		\begin{subfigure}[b]{0.32\textwidth}
			\centering
			\includegraphics[width=\textwidth]{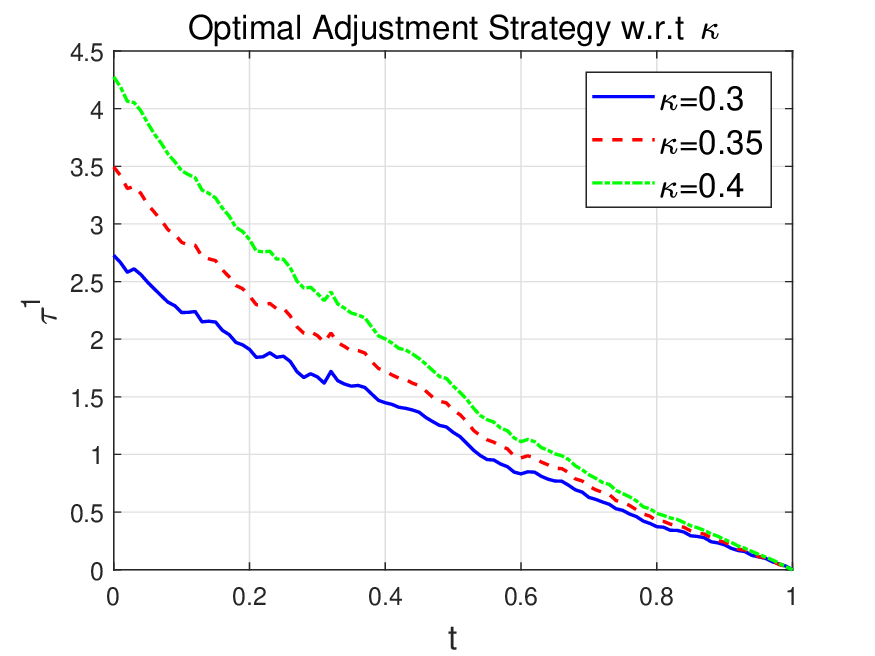}
		\end{subfigure}
		\caption{The strategies of pricing and adjustment w.r.t manufacture cost parameter $\kappa$.}
		\label{fig:kappa}
	\end{figure}
	
	As is shown in Figure \ref{fig:c}, the curve of the strategies of pricing and adjustment changing with $c$ is given. The parameter $c$ represents the cost coefficient for the sellers' adjustment effort $\tau^i$ in their cost term $c(\tau^i)^2$. Since we are considering homogeneous sellers, an increase in $c$ directly led to the decline of $\tau^i$. A higher $c$ significantly raises the marginal cost of adjustment for sellers. This causes a decrease in their optimal adjustment effort  $\tau^i$, as it becomes more expensive to maintain previous activity levels. Anticipating the sellers' reduced responsiveness, the broker and buyer aim to stimulate supply by offering a higher price $p^{{\rm br}}$ and $p^{{\rm bu}}$. 
	\begin{figure}[h]
		\centering
		\begin{subfigure}[b]{0.32\textwidth}
			\centering
			\includegraphics[width=\textwidth]{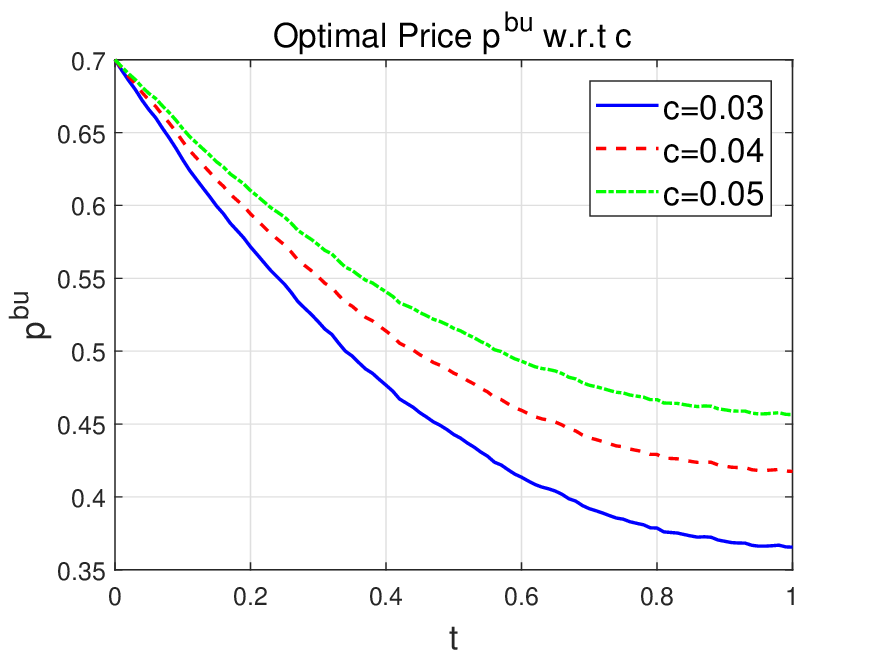}
		\end{subfigure}
		\begin{subfigure}[b]{0.32\textwidth}
			\centering
			\includegraphics[width=\textwidth]{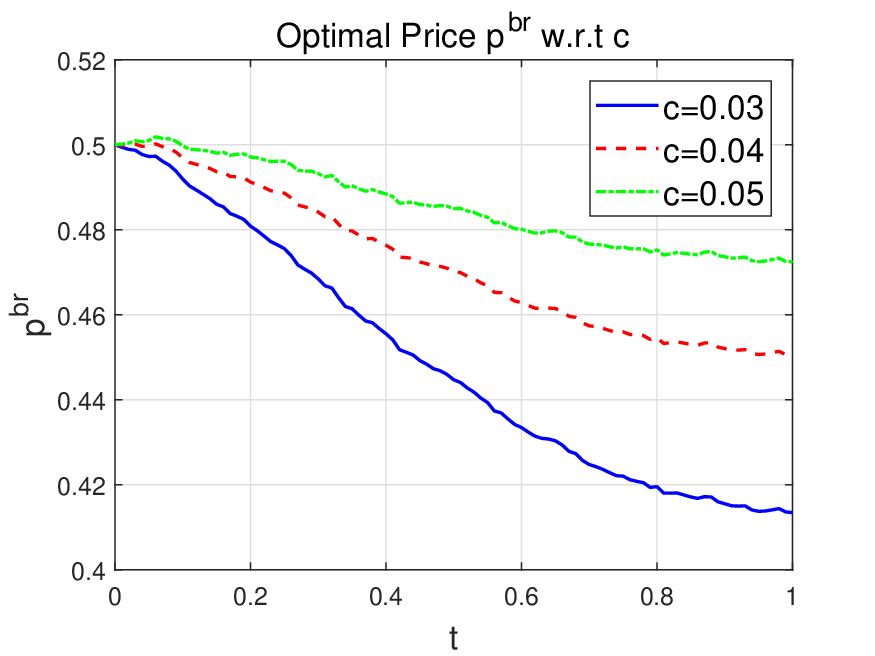}
		\end{subfigure}
		\begin{subfigure}[b]{0.32\textwidth}
			\centering
			\includegraphics[width=\textwidth]{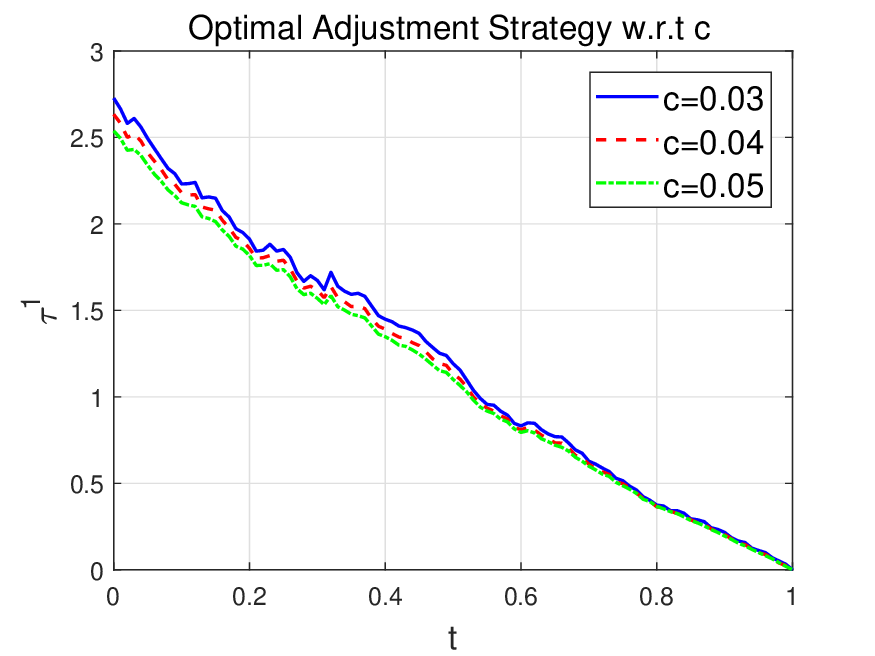}
		\end{subfigure}
		\caption{The strategies of pricing and adjustment w.r.t adjustment cost coefficient $c$.}
		\label{fig:c}
	\end{figure}

\appendix
\section{Proofs of Auxiliary Results}\label{sec:appendix}   	

In this appendix, we provide the proof of the auxiliary lemma:
{\small 
\begin{proof}[Proof of Lemma \ref{lemma3.1}]
Let $C>0$ be a constant depending on $T>0$ only, but may be different from line to line. It follows from \eqref{audynamic} that, for $t\in[0,T]$,
\begin{align*}
\hat{Y}_t^{*, i}=-\int_t^T(b \hat{Q}_s^{*, i}+\alpha  \hat{Y}_s^{*, i}-a p_s^{{\rm br}})ds-\int_t^T Z_s^{*,i}dW_s^i-\int_t^TZ_s^{*, 0, i}dW_s^0.    
\end{align*}
By applying It\^o’s rule to $|\hat{Y}_t^{*,i}|^2$ and 
the inequality $2xy\leq \frac{1}{\varepsilon}x^2+\varepsilon y^2$ for some $\varepsilon>0$, we have, for $t\in[0,T]$,
\begin{align*}
&\mathbb{E}\left[\left|\hat{Y}_t^{*, i}\right|^2+\int_t^T\left|Z_s^{*,i}\right|^2ds+\int_t^T\left|Z_s^{*,0,i}\right|^2ds\right]\leq C\left\{\int_t^T\mathbb{E}\left[|\hat{Q}_s^{*, i}|^2\right]ds+\int_t^T\mathbb{E}\left[|\hat{Y}_s^{*, i}|^2\right]ds+\mathbb{E}\left[\int_t^T|p_s^{{\rm br}}|^2ds\right]\right\}\nonumber\\
&\qquad\qquad\leq C\left\{1+\int_t^T\mathbb{E}\left[|\hat{Q}_s^{*, i}|^2\right]ds+\int_t^T\mathbb{E}\left[|\hat{Y}_s^{*, i}|^2\right]ds\right\}.
\end{align*}
On the other hand, it holds that
\begin{align*}
\sup_{t\in [0,T]}\left|\hat{Y}_t^{*, i}\right|^2\leq C\left\{\int_0^T\left(\left|\hat{Q}_s^{*, i}\right|^2+\left|\hat{Y}_s^{*,i}\right|^2+\left|p_s^{{\rm br}}\right|^2\right)ds+\sup_{t\in [0,T]}\left|\int_t^TZ_s^{*,i}dW_s^i\right|^2+\sup_{t\in [0,T]}\left|\int_t^TZ_s^{*, 0,i}dW_s^0\right|^2\right\}.
\end{align*}
By employing the BDG inequality, we obtain 
\begin{align}\label{est0}
\mathbb{E}\left[\sup_{t\in [0,T]}|\hat{Y}_t^{*, i}|^2\right]&\leq C\left\{\int_0^T\mathbb{E}\left[\sup_{s\in [0,t]}|\hat{Q}_s^{*, i}|^2\right]dt+\int_0^T\mathbb{E}\left[\sup_{s\in [0,t]}|\hat{Y}_s^{*, i}|^2\right]dt+\mathbb{E}\left[\int_0^T|p_t^{{\rm br}}|^2dt\right]\right\}\nonumber\\
&\leq C\left\{1+\int_0^T\mathbb{E}\left[\sup_{s\in [0,t]}|\hat{Q}_s^{*, i}|^2\right]dt+\int_0^T\mathbb{E}\left[\sup_{s\in [0,t]}|\hat{Y}_s^{*, i}|^2\right]dt\right\}.
\end{align}
We also have from \eqref{audynamic} that
\begin{align*}
\hat{Q}_t^{*,i}=q_0+\int_0^t\left(\alpha\left(m_t^*-\hat{Q}_s^{*, i}\right)+\frac{\beta ^2}{2c }\hat{Y}_s^ i\right)ds+\sigma  W_s^i+\sigma_0W_s^0,\quad \forall t\in[0,T].    
\end{align*}
This yields from BDG inequality that
\begin{align}\label{est}
\mathbb{E}\left[\sup_{t\in [0,T]}\left|\hat{Q}_t^{*, i}\right|^2\right]\leq C\left\{1+\int_0^T	\mathbb{E}\left[\sup_{s\in [0,t]}\left|\hat{Q}_s^{*, i}\right|^2\right]dt+\int_0^T\mathbb{E}\left[\sup_{s\in [0,t]}\left|\hat{Y}_s^{*,i}\right|^2\right]dt\right\}.
\end{align}
By Combining \eqref{est0} with \eqref{est}, and using the Gronwall's lemma,  it can be deduced that
$\mathbb{E}[\sup_{t\in [0,T]}|\hat{Y}_t^{*,i}|^2]\leq C$.
By \eqref{ndynamic}, it holds that
\begin{align}\label{est2}
\mathbb{E}\left[\sup_{t\in [0,T]}\left|Q_t^{*, i}\right|^2\right]\leq C\left\{1+\int_0^T	\mathbb{E}\left[\sup_{s\in [0,t]}\left|Q_s^{*, i}\right|^2\right]dt+\int_0^T	\frac{1}{n}\sum_{i=1}^n\mathbb{E}\left[\sup_{s\in [0,t]}\left|Q_s^{*,i}\right|^2\right]dt\right\}.
\end{align}
By averaging $n$ inequalities above, it follows that
\begin{align*}
\frac{1}{n}\sum_{i=1}^n\mathbb{E}\left[\sup_{t\in [0,T]}\left|Q_t^{*,i}\right|^2\right]\leq C\left\{1+\int_0^T	\frac{1}{n}\sum_{i=1}^n\mathbb{E}\left[\sup_{s\in [0,t]}\left|Q_s^{*,i}\right|^2\right]dt\right\}.
\end{align*}
This implies from Gronwall's lemma that $\frac{1}{n}\sum_{i=1}^n\mathbb{E}\left[\sup_{t\in [0,T]}|Q_t^{*, i}|^2\right]\leq C$. Substituting it into \eqref{est2}, the desired result $\mathbb{E}\left[\sup_{t\in [0,T]}|Q_t^{*, i}|^2\right]\leq C$ is obtained .

Recall $m_t^*=\overline{Q}_t^*$ for $t\in[0,T]$. Here, for simplicity, we omit the subscript $m^*$ in $\overline{Q}_t^{m^*,*}$, which corresponds to the mean field term $m^*$ given in Lemma \ref{mfe}. Then $\overline{Q}^*=(\overline{Q}_t^*)_{t\in[0,T]}$ obeys that \begin{align}\label{ccfbsde}
d\overline{Q}_t^{*}=\frac{\beta^2}{2c}\overline{Y}_t^{*}dt+\sigma_0 dW_t^0,~~\overline{Q}_0^{*}=q_0;~~d\overline{Y}_t^{*}=\left(b\overline{Q}_t^{*}+\alpha \overline{Y}_t^{*}- ap_t^{{\rm br}}\right)dt+Z_t^{ *,0}dW_t^0,~~\overline{Y}_T^*=0.
\end{align}
Summing \eqref{ndynamic} from $i=1$ to $n$ and dividing by $n$, we get $d\overline{Q}_t^{*, (n)}=\frac{\beta^2}{2cn}\sum_{i=1}^n\hat{Y}_t^{*, i}dt+\sigma\frac{1}{n}\sum_{i=1}^n dW_t^i+\sigma_0dW_t^0$ with $\overline{Q}_0^{*,(n)}=q_0$. Then, we have
\begin{align}\label{lem2}
\mathbb{E}\left[\sup_{t\in [0,T]}\left|\overline{Q}_t^{*,(n)}-m_t^*\right|^2\right]&\leq C \left\{\int_0^T\mathbb{E}\left[\sup_{s\in [0, t]}\left|\frac{1}{n}\sum_{i=1}^n \hat{Y}_s^{*,i}-\overline{Y}_t^*\right|^2\right]dt+\mathbb{E}\left[\sup_{t\in [0,T]}\left|\frac{1}{n}\sum_{i=1}^ndW_t^i\right|^2\right]\right\}\nonumber\\
&\leq C \left\{\int_0^T\mathbb{E}\left[\sup_{s\in [0, t]}\left|\frac{1}{n}\sum_{i=1}^n \hat{Y}_s^{*,i}-\overline{Y}_t^*\right|^2\right]dt+\frac{1}{n}\right\}.
\end{align}
Decoupling \eqref{ndynamic} by following the same procedure as for \eqref{fbsde3.1}, we have $\hat{Y}_t^{*,i}=\xi_t(\hat{Q}_t^{*,i}-m_t^*)+\vartheta_tm_t^*+\zeta_t$,  $Z_t^{*,i}=\sigma\xi_t$, $Z_t^{*,0,i}=\sigma_0\vartheta_t+\sigma_t^{\zeta}$,
and $d\hat{Q}_t^{*,i}=\left(\left(\alpha-\frac{\beta^2}{2c}\xi_t\right)\left(m_t^*-\hat{Q}_t^{*,i}\right)+\frac{\beta^2}{2c}\left(\vartheta_tm_t^*+\zeta_t\right)\right)dt+\sigma dW_t^i+\sigma_0dW_t^0$. Taking conditional expectation and taking the difference, for $i=1,\ldots,n$, 
\begin{align*}
d\left(\overline{\hat{Q}_t^{*,i}}-\overline{Q_t^*}\right)=\left(\alpha-\frac{\beta^2}{2c}\xi_t\right) \left(\overline{Q_t^*}-\overline{\hat{Q}_t}^{*,i}\right)dt,\quad \overline{\hat{Q}_0^{*,i}}-\overline{Q_0^*}=0.    
\end{align*}
Consequently, a.s. $\overline{\hat{Q}_t^{*,i}}:=\mathbb{E}[\hat{Q}_t^{*,i}|\mathcal{F}_t^0]=\overline{Q_t^*}=m_t^*$ and $\overline{\hat{Y}_t^{*,i}}:=\mathbb{E}[\hat{Y}_t^{*,i}|\mathcal{F}_t^0]=\overline{Y}_t^*$ for $t\in[0,T]$. Note that $\mathbb{E}\left[\sup_{t\in[0, T]}\left|\frac{1}{n}\sum_{i=1}^n\hat{Y}_t^{*,i}-\overline{Y}_t^*\right|^2\right]\leq \frac{1}{n}\mathbb{E}\left[\sup_{t\in[0, T]}\left|\hat{Y}_t^{*,1}-\mathbb{E}[\hat{Y}_t^{*,1}|\mathcal{F}_t^0]\right|^2\right]=O(n^{-1})$. Substituting it into \eqref{lem2}, we have $\mathbb{E}\left[\sup_{t\in [0,T]}\left|\overline{Q}_t^{*,(n)}-m_t^*\right|^2\right]=O(n^{-1})$. Thus, the proof of the lemma is complete.
\end{proof}}
    
\end{document}